\documentclass[11pt,reqno]{amsart}
\usepackage{geometry}
\usepackage[utf8]{inputenc}
\usepackage{amsfonts,epsfig,fancyhdr}
\usepackage{amsmath,amssymb,amsthm, fancyhdr,latexsym,graphicx,graphics}
\usepackage{multicol,picinpar}
\usepackage{pdfsync}
\usepackage{euscript}
\usepackage{wrapfig}
\usepackage{caption}
\usepackage{tcolorbox}
\numberwithin{equation}{section}

\newtheorem{theorem}{Theorem}[section]
\newtheorem{cor}[theorem]{Corollary}
\newtheorem{defi}[theorem]{Definition}
\newtheorem{lemma}[theorem]{Lemma}
\newtheorem{rem}[theorem]{Remark}

\newtheorem{prop}[theorem]{Proposition}

\newtheorem{Mlemma}[theorem]{Main Lemma}

\theoremstyle{definition}
\newtheorem{definition}{Definition}[section]

\usepackage{enumerate}

\usepackage{euscript,color}

\definecolor{red}{rgb}{1,0,0}

\begin{document}
	
	\title[Normal holonomy of complex hyperbolic submanifolds]
	{Normal holonomy of complex hyperbolic submanifolds}

	\author[S. Castañeda-Montoya]{Santiago Castañeda-Montoya}

	\author[C. Olmos]{Carlos E. Olmos}

        \address{\ \newline
        FAMAF\\ Universidad Nacional de Córdoba\\ \newline ciudad universitaria\\ (5000) Córdoba\\ Argentina.}

        \email{santiago.castaneda@mi.unc.edu.ar}
        \email{olmos@famaf.unc.edu.ar|}
	
	%\keywords{normal holonomy, CR-submanifolds, normal connection, s-representations.}
	\thanks{Both authors were supported by FaMAF-UNC and Ciem-CONICET}
	
	%\date{\today}
\begin{abstract}
We prove that the restricted normal holonomy group of a Kähler submanifold of the complex hyperbolic space $\mathbb{C}H^{n}$ is always transitive, provided the index of relative nullity is zero. This contrasts with the case of $\mathbb{C}P^{n}$, where a Berger type result was proved by Console, Di Scala, and the second author. The proof is based on lifting the submanifold to the pseudo-Riemannian space $\mathbb{C}^{n,1}$ and developing new tools to handle the difficulties arising from possible degeneracies in holonomy tubes and associated distributions. In particular, we introduce the notion of weakly polar actions and a framework for dealing with degenerate submanifolds. These techniques could contribute to a broader understanding of submanifold geometry in spaces with indefinite signature, offering new insight into submanifolds in the dual setting of complex projective geometry.
\end{abstract}

	\maketitle

    \section{Introducion}  

For submanifolds of spaces of constant curvature, a fundamental result is the so-called normal holonomy theorem \cite{O1}. It states that the representation of the restricted normal holonomy group on the normal space is, up to a trivial factor, equivalent to an $s$-representation (i.e., the isotropy representation of a semisimple symmetric space). This result is an important tool for studying submanifold geometry, particularly for submanifolds with simple geometric invariants, such as isoparametric and homogeneous submanifolds. Moreover, there is a subtle interplay between Riemannian and normal holonomy which has led to a geometric proof of the Berger holonomy theorem \cite{O2} (for a general reference on this topic, see \cite{BCO}).
The normal holonomy theorem was extended to Kähler submanifolds of the complex space forms $\mathbb{C}P^n$ and $\mathbb{C}H^n$ by Alekseevsky and Di Scala \cite{AD}. They proved that if the normal holonomy representation is irreducible, then it is a Hermitian $s$-representation. In the reducible case, up to multiplication by complex numbers of unit modulus, it is still a Hermitian $s$-representation.
Moreover, they showed that the normal holonomy representation is always irreducible when the index of relative nullity is zero. In this context, one has a Berger type holonomy theorem \cite{CDO}: {\it a complete, full complex submanifold  of $\mathbb{C}P^{n}$ with a non-transitive normal holonomy group is the complex orbit, in the projective space, of an irreducible Hermitian $s$-representation} (see \cite{DV} for a generalization). In fact, the assumption of completeness is used only to guarantee, by a result of Abe and Magid \cite{AM}, that the index of relative nullity is zero. The main techniques consisted of taking the canonical lift of the submanifold to $\mathbb{C}^{n+1}$ and using methods from submanifold geometry.

The main purpose of this article is to address the natural question of whether the aforementioned results can be extended to complex submanifolds of complex hyperbolic space $\mathbb{C}H^n$. To this end, we lift the complex submanifold to $\mathbb{C}^{n,1}_{-}$, the open subset of $\mathbb{C}^{n,1} \simeq \mathbb{C}^{n+1}$ consisting of vectors $v$ satisfying $\langle v, v \rangle < 0$, where $\langle\,  ,\, \rangle$ denotes the Hermitian inner product of complex signature $(n,1)$.
The main challenge stems from the fact that the geometry of submanifolds in pseudo-Riemannian spaces is significantly more intricate, primarily due to the possible degeneracy — in the sense that the ambient metric may restrict degenerately — of holonomy tubes and of the equivalence classes defined by certain distributions. To tackle this issue, we first introduce the concept of weakly polar actions and develop a geometric framework for dealing with degenerate submanifolds. Although the normal connection is not well-defined for such submanifolds, the notion of a parallel normal field remains meaningful.

    \begin{theorem}\label{thm:MainTheorem1}
Let $\bar{N}^n$ be a full Kähler submanifold of the complex hyperbolic space 
$\mathbb{C}H^{n+k}$ with zero index of relative nullity. Then the restricted normal holonomy group $\Phi$ is transitive (or equivalently, $\Phi \simeq U_k$, since it acts as a Hermitian $s$-representation).
\end{theorem}

Let us note that when the index of relative nullity is non-zero, the normal holonomy group representation may be reducible. For example, if $M$ is a complex submanifold of $\mathbb{C}H^m$ and $N$ is a complex submanifold of $\mathbb{C}P^n$, then the open subset $\mathcal{O}$ of negative points of the abstract join $J(M, N)$ forms a complex submanifold of $\mathbb{C}H^{m+n+1}$ whose normal holonomy group is reducible.

We hope that the techniques developed in this paper will be useful for studying submanifolds in spaces with indefinite signature, with a focus on normal holonomy.
\vspace{.15cm}

The paper is organized as follows. Section~2 contains the preliminaries and basic facts necessary for our purposes. In this section, we develop general tools that may also be useful in a broader context. We begin with standard results on the adapted normal curvature tensor in $\mathbb{R}^{r,s}$, reviewing in \S2.1 known facts about isoparametric submanifolds in Lorentz space. In \S2.2, we define the concept of an \emph{essentially Riemannian submanifold}, and in \S2.3, we prove a normal holonomy theorem for such submanifolds. In \S2.4, we define \emph{weakly polar actions}, without requiring that the maximal dimensional orbits be non-degenerate. The main general result is Proposition~\ref{prop:Po}, which is applied in Theorem~\ref{thm:NholPo} to the study of normal holonomy. This, in turn, is used to extend normal vectors to parallel normal fields, possibly in a degenerate context. In \S2.5 it is extended the theory of holonomy tubes of Euclidean submanifolds to $\mathbb{R}^{r,s}$, even in degenerate cases. In \S2.6, we define the \emph{horosphere embedding}, which will play a crucial role in the focalization of the $0$-eigendistributions associated with parallel normal fields. This may be regarded as a focalization at infinity.

 Section~3 is concerned with the lift of complex submanifolds of $\mathbb{C}H^n$ to $\mathbb{C}^{n,1}$, relating the respective normal holonomy groups and relative nullity distributions.

 Section~4 is concerned with generalized holonomy tubes and their  relation with the so-called \emph{canonical foliation}, extending arguments in \cite{CDO}. The delicate point is the proof of Main Lemma \ref{lem:main343}.

In Section~5, we study the geometry of the equivalence classes of the distribution perpendicular to the nullity. Coxeter groups are defined, inspired by Terng's construction of such groups for isoparametric submanifolds. This section includes the proof of Theorem~\ref{thm:MainTheorem1}.

	\section{Preliminaries and basic facts}
	
	Let $\mathbb V$ be a real vector space of dimension $n$ and let 
	$\langle \, , \, \rangle$ be an inner product of signature $(r,s)$, where $n=r+s$ with $s$ being the dimension of a maximal negatively definite subspace of $\mathbb{V}$. We will often refer to $s$ as the signature of $\mathbb V$, when the inner product is clear from the context. As usual, $\mathfrak {so}(\mathbb V)$ denotes the Lie algebra of the skew-symmetric (i.e. anti self-adjoint) endomorphisms of $(\mathbb V, \langle \, , \, \rangle )$.
	The inner product induces an inner product, also denoted by $\langle \, , \,\rangle$, on  tensors of a fixed type. In particular,  $\langle x\otimes  y , w\otimes z\rangle = 
	\langle x,w\rangle\langle y , z\rangle$.
	
	We focus on the inner product induced on  $\Lambda ^2 (\mathbb V)$. On has that  
	\begin{equation}\langle x\wedge y , w\wedge z\rangle = 
		2(\langle x,w\rangle \langle y,z\rangle -
		\langle x,z\rangle) \langle y,w\rangle), 
	\end{equation}\label{eq:1-1}
	where $u\wedge v = u\otimes v -v\otimes u$.

	If $e_1, \cdots , e_n$ is an orthonormal basis of $\mathbb V$, then $\frac{1}{\sqrt{2}}\, e_i\wedge e_j = \frac{1}{\sqrt{2}}(e_i\otimes e_j -e_j\otimes e_i)$, is an orthonormal basis of $\Lambda ^2(\mathbb V)$, $i,j= 1, \dots , n$, $i<j$. 
	
	The vector space $\Lambda ^2 (\mathbb V)$ is naturally identified with 
	$\mathfrak {so}(\mathbb V)$ by means of 
	\begin{equation}\label{eq:1-2}
		\ell : \Lambda ^2 (\mathbb V) \to \mathfrak {so}(\mathbb V), 
	\end{equation}
	where $\ell$ is determined by  
	\begin{equation}\label{eq:ell11} \langle \ell (x\wedge y)w, z\rangle = \langle x , w\rangle
		\langle y , z\rangle - \langle x , z\rangle
		\langle y , w\rangle.\end{equation} Observe that 
	\begin{equation} \label{eq:ell-12}
		\langle \ell (x\wedge y)w, z\rangle = \frac 12 \langle x\wedge y , w\wedge z\rangle 
	\end{equation}
	
	Endow $\mathfrak {so}(\mathbb V)$ with the usual inner product
	$$\langle B, C\rangle = -\mathrm{trace}(B\circ C).$$
	A straightforward calculation shows that 
	$$\langle \ell ( x\wedge y),\ell (w\wedge z)\rangle =
	\langle x\wedge y , w\wedge z\rangle $$ 
	which implies that $\ell$ is a linear isometry.

	One has, from (\ref{eq:ell-12}),   that 
	\begin{equation}\label{eq:ell-0}
		\langle\ell^{-1}(B), w\wedge z\rangle
		=  2\,  \langle Bw ,z\rangle  
	\end{equation}
	and hence 
	\begin{equation}\label{eq:ell-1}
		\langle\ell^{-1}(B), \frac{1}{\sqrt{2}}\, e_i\wedge e_j\rangle
		=  {\sqrt{2}}\,  \langle Be_i ,e_j\rangle,  
	\end{equation}
	and hence 
	\begin{equation}\label{eq:ell-123}
		\ell^{-1}(B)=  \sum _{i<j}\epsilon _i \epsilon _j \langle Be_i, e_j\rangle \,  e_i\wedge e_j,
	\end{equation}
	where $\epsilon _k = \langle e_k,e_k\rangle = \pm  1$.

    \vspace{.2cm}
    
    In order to fix notation, since the word {\it degenerate} is ambiguous, we explicit the following definition: 
	\begin{definition}
     A (regular) submanifold of a pseudo-Riemannian manifold is called {\it degenerate} if the induced metric is degenerate.
	\end{definition}
	
Let \( M^{k,l} \subset \mathbb{R}^{r,s} \) be a non-degenerate local submanifold of the flat space form of signature \( s \) and dimension \( n = r + s \).  Here, {\it local} means that we work in a neighborhood of a point, without any global assumptions. 
	Let us consider the normal curvature tensor $R^\perp$ at some arbitrary $q\in M$. Recall the Ricci identity $\langle R^\perp_{x,y}\xi,\eta\rangle = \langle [A_{\xi}, A_{\eta}]x,y\rangle$, where $A$ is the shape operator of $M$. 
	
	Just for the sake of saving notation, we use the same letter $\ell$ for the isometry 
	$\ell : \Lambda ^2 (\mathbb V) \to \mathfrak {so}(\mathbb V)$, where $\mathbb V$ is either $T_qM$ or $\nu _qM$. 
	Let $x,y\in T_qM$ and $\xi , \eta \in \nu _qM$ be arbitrary. Since $\ell ^{-1}(R_{x,y})$ is skew-symmetric in $x,y$ it extends to a linear map $\tilde R ^\perp: \Lambda ^2(T_qM) \to \Lambda ^2(\nu _qM)$, by defining 
	\begin{equation}\label{eq:tilde}
		\tilde R ^\perp (x\wedge y) = \ell^{-1}(
		R^\perp _{x,y}).
	\end{equation}
	We will refer to  $\tilde R ^\perp$ as the normal curvature operator. 
	
	\begin{align}
		\begin{split}
			\frac 12 \langle \tilde R^\perp (x\wedge y) , \xi \wedge \eta\rangle &=  \langle R^\perp_{x,y}\xi , \eta \rangle = 
			\langle [A_{\xi}, A_{\eta}]x,y\rangle \\ &= \frac 12
			\langle\ell ^{-1}([A_{\xi}, A_{\eta}]), x\wedge y\rangle = 
			\frac 12 \langle \tilde A (\xi \wedge \eta), x\wedge y\rangle
		\end{split}\label{eq:132}
	\end{align}
	where $\tilde A : \Lambda ^2(\nu _qM)\to \Lambda ^2(T_qM)$ is the linear map defined by $\tilde A (\xi\wedge \eta) = \ell ^{-1}([A_{\xi}, A_{\eta}])$. Then
	\begin{equation}\label {eq:2121}
		\langle \tilde R^\perp (x\wedge y) , \xi \wedge \eta\rangle 
		= \langle  x\wedge y ,\tilde A (\xi \wedge \eta)\rangle
	\end{equation}
	This implies that $\tilde A$ is the transpose morphism  of the normal curvature operator  $\tilde R^\perp$ (or,equivalently, $\tilde R^\perp$ is the transpose of $\tilde A$).
	
	\begin{rem}\label{rem:11} \rm
		Let us define the so-called {\it adapted} normal curvature tensor: 
		\begin{align*}
			\begin{split} 
				\langle \mathcal R_{\xi_1,\xi _2}\xi _3,\xi _4\rangle :&=\langle \tilde R \circ \tilde A \, (\xi _1 \wedge \xi _2), \xi _3\wedge 
				\xi _4\rangle 
				= \langle  \tilde A \, (\xi _1 \wedge \xi _2), \tilde A \, (\xi _3\wedge 
				\xi _4)\rangle \\ 
				&= \langle \ell ^{-1}([A_{\xi _1}, A_{\xi _2}]), 
				\ell ^{-1}([A_{\xi _3}, A_{\xi _4}])\rangle  \\
				&= \langle [A_{\xi _1}, A_{\xi _2}],
				[A_{\xi _3}, A_{\xi _4}]\rangle  
				= -\mathrm{trace}\, ([A_{\xi _1}, A_{\xi _2}]\circ [A_{\xi _3}, A_{\xi _4}])
			\end{split}
		\end{align*}
		Then, by the same arguments  in \cite{O1}, $\mathcal R$ satisfy the identities of a pseudo-Riemannian curvature tensor on the normal space $\nu _q(M)$.
	\end{rem}
	
	\begin{lemma} \label{lem:11} Let $\mathbb V$ be a vector space  with a positive definite inner product and let  $\mathbb W$ be a vector space  with an inner product. We denote both inner products by $\langle \, ,\rangle$. 
		Let $L: \mathbb V \to \mathbb W$ be a linear map and let 
		$L^t: \mathbb W \to \mathbb V$ be its transpose. Then the image of $L$ coincides with the image of $L\circ L^t$.
	\end{lemma}
	\begin{proof} The inclusion $L\circ L^t(\mathbb W) \subset L(\mathbb V)$ is clear. If $\mathbb V' = L^t(\mathbb W)^\perp$, then 	
		$$\{0\}= \langle  \mathbb V', L^t(\mathbb W)\rangle = 
		\langle  L(\mathbb V'), \mathbb W \rangle .$$
		Thus,  $L(\mathbb V')= \{0\}$. Since $\mathbb V= 
		\mathbb V'\oplus L^t(\mathbb W)$ the lemma follows.
	\end{proof}
	
	\begin{rem}\label{rem:32} \rm
		In the notation and assumptions of the Lemma \ref{lem:11}, let $C:=L\circ L^t$. Then 
		$\langle C(w), w\rangle = \langle L^t(w), L^t(w)\rangle \geq 0$ with equality  if and only if $L^t(w)=0$. 
	\end{rem}
	\subsection{Riemannian isoparametric submanifolds of the Lorentz space}\label{sec:isoLorentz}
	
	The object of this section is to point out some local results that  in the bibliography are only proved for  complete submanifolds 
 (see \cite{Wu},\cite{Wi}, \cite[Section 4.2.6]{BCO}). 
	
	Let $M^n$ be a local isoparametric Riemannian submanifold of Lorentz space $\mathbb R^{m,1}$.  Namely, $M$ is a local Riemannian submanifold with (globally) flat normal bundle and, the shape operator $A_\xi$ has constant eigenvalues for any parallel normal section $\xi$. As in the Euclidean ambient case we have an orthogonal  decomposition 
	$TM= E_0\oplus \cdots \oplus E_g$, perhaps where $E_0=0$, 
	and  different parallel normal fields, known as curvature normals,  
	$0=\eta _0, \cdots , \eta _g$
	such that any of the so-called eigendistributions $E_i$ is invariant under all the shape operators of $M$ and 
       $$A_{\xi\vert E_i}= \langle \eta _i , \xi\rangle
	Id_{\vert E_i}.$$	One has, due to Codazzi identity, that any eigendistribution is autoparallel. Moreover, the integral manifold $S_i(x)$ of $E_i$ by $x$ is an umbilical submanifold of the ambient space, which is contained in the affine subspace $$ L_i(x) = x + E_i(x) \oplus \mathbb R\eta _i(x)
	\subset \mathbb R^{m,1}.$$ 
	It turns out that  $L_i(x) = L_i(y)$
	if $S_i(x)=S_i(y)$, $i=0, \cdots , g$. Let $k_i=\dim E_i$.
	One has, for $i=0$, that $S_0(x)$ is an open part of $L_0(x) = x + E_0(x)$. 
	If $i>1$, then $S_i(x)$ is an umbilical hypersurface of $L_i(x)$ that belongs to   one of the following types: 
	
	(1)   If $\eta _i$ is spacelike, then $S_i(x)$ is an open subset of the round $k_i$-sphere of $L_i(x)$ of center $c$ and radius $\rho$ given by 
	\begin{equation}\label{eq:99t} c= x + \frac {1}{\langle \tilde \eta _i(x), \tilde \eta _i(x)\rangle}\, \tilde \eta _i\, ,
		\ \ \ \ \ \ \
		\rho ^ 2= \frac {1}{{\langle \tilde \eta _i(x), \tilde \eta _i(x)\rangle}}.
	\end{equation} 
    In this case the geodesics of $S_i(x)$ are circles 
	
	(2) If $\eta _i$ is timelike then $S_i(x)$ is an open subset of the hyperbolic space  of $L_i(x)$ defined by 
	$$H_r^{k_i} = \{X\in x +  E_i(x) \oplus \mathbb R \,  \eta _i(x):  
	\langle X-c, X-c)\rangle = -r^2 \}^o$$, where
	 	$-r^2= \langle x-c,x-c\rangle $,
 	 $c$ has the same expression as in (1) and  $(\ )^o$ denotes the connected component by $x$. 
     In this case the geodesics of $S_i(x)$,  are of the form 
     \begin{equation}\label{eq:438}
         -\cosh (t)\eta _i(x) + \sinh (t) w, 
     \end{equation}  where 
     $w\in T_xS_i(x)$ and $\langle w , w \rangle = r^2 =-\langle \eta _i(x) , \eta _i(x) \rangle $
	
	(3) If $\eta _i$ is lightlike, then $S_i(x)$ is a horosphere of an appropriate real hyperbolic space. In fact, there always exist a timelike $z\in \nu_xM$ such that 
	$\langle  \eta _i(y) , z\rangle =1$. Extend $z$ to a parallel normal field $\tilde z$ of $S_i(x)$. Then the shape operator $A_{\tilde z}$ is  the identity, since $\langle  \eta _i , \tilde{z}\rangle =1$. Then the image of the parallel map $y \mapsto y + \tilde z _y$, from $S_i(x)$ into $\mathbb R^{m,1}$, is  a constant $c =x+z$, since its differential is zero. Then $\langle y-c ,y-c\rangle = \langle \tilde z, \tilde z \rangle: = -r^2$, for all $y\in S_i(x)$.
Let 
	$$H_r^{k_i+1} = \{X\in x +  E_i(x) \oplus \mathbb R \,  \eta _i(x) \oplus \mathbb R\, z:  
 \langle X-c, X-c\rangle = -r^2\}^o.$$
	 Then $S_i(x)$ is an open subset of the horosphere defined by 
	$$\bigg(x +  E_i(x) \oplus \mathbb R \, \eta _i(x)\bigg)\cap  H_r^{k_i+1}\, .$$

    Any component  of a geodesic $\gamma (t)$ of $S_i(x)$ is  quadratic, i.e.  of the form $a_1t^2 + a_2t + a_t$ (see Section \ref{sec:horosphere}, and \cite{Wi} for an explicit expression).

 \ 
	
		\begin{prop}\label{prop:sl-nsl}
	Let $M$ be a Riemannian isoparametric submanifold of the Lorentz space $\mathbb R^{m,1}$. Then any non-space like curvature normal $\eta _i$ is perpendicular to any other curvature normal. 
	\end{prop}
\begin{proof} We may assume that $\eta _i \neq 0$.  Let $S_i(p)$ be an integral manifold of $E_i$,  let $M_i : = (M)_{\xi _i}$ be a parallel focal manifold such that $\ker (I-A_{\xi _i}) = E_i$, and let $\pi$ be the projection from $M\to M_i$,  i.e. $\pi (q) = q +\xi _i(q)$. 
Since 
$\langle \eta _i,\eta _i\rangle\leq 0$, then  $S_i(p)$ is a open subset of an unbounded complete Riemannian  umbilical submanifold $\tilde S _i(p)$ of the Lorentzian affine normal space $\pi (p) +\nu _{\pi (p)}M_i \subset \mathbb R^{m,1}$ (see the discussions preceding this proposition). By the tube formula, the eigenvalues of the shape operator $A^i_{q-\pi (p)}$ of $M_i$ do not depend on $q\in S_i(p)$ and hence, 
$d: =\Vert A^i_{q-\pi (p)}\Vert$ does not depend on $q\in S_i(p)$. Let $\gamma (t)$, $\vert t\vert <\varepsilon$ be a geodesic of $S_i(p)$ and let $\tilde \gamma (t)$ be its extension to a complete geodesic of 
$\tilde S _i(p)$, $t\in \mathbb R$. We have that $\Vert A^i_{\gamma (t)-\pi (p)}\Vert =d$, for $\vert t\vert <\varepsilon$. By the explicit form  of the geodesics, and by making use standard arguments relying  on the (real)  analyticity of $\tilde \gamma (t)$, we obtain that 
$\Vert A^i_{\tilde\gamma (t)-\pi (p)}\Vert =d$, for all $t\in \mathbb R$. Then the image of $\tilde \gamma (t)$ under the affine map $u\mapsto A^i_{u-\pi (p)}$, from  $\pi (p) +\nu _{\pi (p)}M_i$ into the symmetric endomorphisms of $T_{\pi(p)}M_i$, is bounded.  This is a contradiction, from the explicit expression of $\tilde \gamma (t)$, unless $A^i_{\tilde\gamma (t)-\pi (p)}$ is constant and thus, $A^i_{q-\pi (p)}$ does not depend on $q\in S_i(p)$ (cf. \cite[Lemma 4]{Wi} or \cite[Lemma 4.2.19]{BCO}).
Then, by the proof of Lemma 4.2.20 of \cite{BCO}, we obtain that $E_i$ is a parallel distribution of $M$. Let 
$v_i$ and $v_j$ of unit length and tangent to $E_i(p)$ and $E_j(p)$, respectively ($i\neq j$). By making use of the Gauss equation, taking into account that $E_i$ is a parallel distribution and that $\alpha (E_i,E_j)=0$, we obtain that
$\langle R(v_1,v_2)v_1,v_2 \rangle=\langle \alpha(v_1,v_2),\alpha(v_1,v_2)\rangle - \langle \alpha(v_1,v_1),\alpha(v_2,v_2)\rangle =-\langle\eta _i(p), \eta _j(p)\rangle= 0$.
	\end{proof}

	\subsection{Essentially Riemannian submanifolds}\label{sec:essR}
	
	\begin{defi}\label{def:1}\rm  A non-degenerate (immersed) submanifold $M^{k,l}$ of $\mathbb R ^{r,s}$ is called {\it essentially Riemannian}
		if there exists a  distribution $\mathcal D$ on $M$, where $\langle\, ,\, \rangle$ is positive definite, such that $\mathcal D$ is invariant under all shape operators $A_\xi$, and the family of the shape operators, restricted to $\mathcal D^\perp$, is a commuting family. \rm 
	\end{defi}
	
	Let  $M^{k,l}$  be an essentially Riemannian submanifold of 
	$\mathbb R ^{r,s}$ with associated Riemannian distribution $\mathcal D$. Since we will work locally, we assume that 
	$M\subset \mathbb R ^{r,s}$ is an embedded submanifold. Let, for $q\in M$, 
	\begin{equation} \label{eq:33}
		C_q = \{R^\perp_{x_q,y_q}: x_q,y_q\in T_qM\}
	\end{equation}
	
	By the Ricci identity, and the fact that the family of shape operators restricted to $\mathcal D^\perp$ is a commuting family one has that 
	\begin{equation} \label{eq:3}
		C_q = \{R^\perp_{x_q,y_q}: x_q,y_q\in \mathcal D_q\}
	\end{equation}
	
	Since the bracket of any two shape operators $[A_\xi , A_\eta]$ is zero when restricted to $\mathcal D^\perp $, and  the restriction of $\langle \, ,\, \rangle $ to $\mathcal D$ is positive definite, one obtains, from Remark \ref{rem:11}, Lemma \ref{lem:11} and Remark \ref{rem:32}, the following: 
	\begin{lemma}\label{lem: 43} Let $M$ be an essentially Riemannian submanifold of $\mathbb R^{r,s}$ and let  $\mathcal R$ be its adapted normal curvature tensor. Then 
		\begin{enumerate}
			\item 	$\mathcal R$
			has non-positive sectional curvatures, i.e.   
			$\langle \mathcal R _{\xi_1,\xi_2}\xi _2 , \xi_1\rangle \leq 0$, for all $\xi _1,\xi _2\in \nu _qM, q\in M$. 
			\item  $\langle \mathcal R _{\xi_1,\xi_2}\xi _2 , \xi_1\rangle = 0$ if and only if $\mathcal R _{\xi_1,\xi_2}= 0$.
			\item $\langle \mathcal R _{\xi_1,\xi_2}\xi _2 , \xi_1\rangle = 0$ if and only if $[A_{\xi_1}, A_{\xi_2}]=0$.
			\item The linear span of $\{R^\perp _{x,y}: x,y\in T_qM\}$
			coincides with the linear span of $\{\mathcal R _{\xi,\eta}: \xi ,\eta \in \nu_qM\}$
			\item Let $\bar{\mathcal R}= \tilde R^\perp \circ \tilde A$ be the curvature operator on $\nu _qM$ associated to $\mathcal R$. Then $\langle \bar{\mathcal R}(u), u\rangle\geq 0$ for all $u\in \Lambda ^2(\nu_qM)$. Moreover, the equality holds if and only if 
			$\bar{\mathcal R}(u)=0$.
		\end{enumerate}
	\end{lemma}

	\subsection{Normal holonomy of essentially Riemannian submanifolds}
	
	\
	
	Let $M$ be an essentially Riemannian submanifold of $\mathbb R^{r,s}$ with adapted normal curvature tensor $\mathcal R$. 
	Let $\tau _c^\perp$ denote the $\nabla ^\perp$-parallel transport along a (piecewise differentiable) curve $c$ from $p$ to $q$, and let $\tau _c (\mathcal R)$ be  the algebraic curvature tensor of $\nu _qM$ defined by 
	$$\tau _c (\mathcal R) _{\xi_1 ,\xi _2}\xi _3 := \tau _c \mathcal R _{\tau _c^{-1}\xi_1 ,\tau _c^{-1}\xi _2}\tau _c^{-1}\xi_3.$$
	Let $\mathfrak {hol}(q)$ denote the normal holonomy algebra at $q$, i.e., the Lie algebra of the normal holonomy group $\Phi (q)$ of $M$ at $q$.
	Then, by the Ambrose-Singer theorem and Lemma \ref{lem: 43}(4), one has that 
	\begin{equation}\label{eq:hol}
		\mathfrak {hol}(q)= \text{linear span of }\{R_{\xi ,\eta }: R\in F(q),\,  \xi , \eta \in \nu _qM\}
	\end{equation}
	where 
	\begin{equation}\label{eq:Fq}
		F(q):=\{\tau _c(\mathcal R_x): c \text{ is an arbitrary curve from } x\text{ to } q, x\in M\}.
	\end{equation}
	Observe that any  $R\in F (q)$ is an algebraic curvature tensor of 
	$\nu _q (M)$. Moreover, it is positive semi-definite when regarded as a symmetric endomorphism of $\Lambda ^2(\nu _qM)$. That is, if $u \in \Lambda ^2(\nu _qM)$, then $\langle R(u),u\rangle \geq 0$ with equality if and only if $R(u)=0$, 
	
	\begin{lemma}\label{lem:unR} We are under the previous notation and assumptions. There exists $R\in  F(q)$ such that  $\displaystyle{\mathfrak {hol}(q)= \{R_{\xi ,\eta }:  \xi , \eta \in \nu _qM\}}$.
	\end{lemma}
	\begin{proof}
		Any $R\in \mathcal F(q)$ will be regarded as a symmetric  endomorphism of 
		$\Lambda ^2(\nu _qM)\underset{\ell}{\simeq} \mathfrak {so}(\nu _qM)$. By means of  this identification  (\ref{eq:hol}) is equivalent to   
		\begin{equation}\label{eq:hol2}
			\ell (\mathfrak {hol}(q))= \text{linear span of }\{\mathrm{Im}(R): R\in  F(q)\},
		\end{equation}
		where $\mathrm{Im}$ denotes the image. 
		
		One has that 
		$\mathrm{Im}(R)^\perp = \ker(R)$, for all $R\in  F (q)$. 
		If $R, R'\in  F(q)$, then 
		$$(\mathrm{Im}(R) + \mathrm{Im}(R'))^\perp = \ker (R)\cap \ker (R').$$
		Let $u\in \Lambda ^2(\nu_qM)$ that belongs to $\ker (R+R')$. Then 
		$$0= \langle (R+R')(u), u\rangle = \langle R(u), u\rangle + 
		\langle R'(u), u\rangle.$$
		Since both $\langle R(u), u\rangle$ and  $\langle R'(u), u\rangle$ are non-positive, then $\langle R(u), u\rangle  =  \langle R'(u), u\rangle = 0$. Then, by Lemma \ref{lem: 43} (5), 
		$R(u)= R'(u) = 0$. Then $\ker (R+R')\subset \ker (R)\cap \ker (R')$. Since the other inclusion is trivial we obtain the equality. Hence, $$\mathrm{Im}(R+R')=  \ker (R+R')^\perp = 
		(\ker (R)\cap \ker (R'))^\perp = \mathrm{Im}(R) + \mathrm{Im}(R').$$ 
		Since $\ell (\mathfrak {hol}(q))$ is the sum of the images of a finite number of elements of $F(q)$, by making use of the previous argument, we conclude the proof.
	\end{proof}
	
	Let us recall the concept of weak irreducibility. Let $\mathbb V$ be a vector space endowed with an inner product $\langle \, ,\, \rangle$, with signature $s$, and let $G$ be a Lie subgroup of $SO(\mathbb V, \langle \, ,\, \rangle)$. We say that $G$ acts on $\mathbb V$ {\it weakly irreducibly} if any $G$-invariant proper subspace of $\mathbb V$ is degenerate (i.e., $\langle \, ,\, \rangle$ is degenerate on $\mathbb V$). 
	
	With the same proof as in \cite {O1} (see also Section 3.3 of  \cite{BCO}) we have the following:

	\begin{prop}\label{prop:NHT}
		Let  $\Phi(q)$ be the  restricted normal holonomy group at $q$ of an essentially Riemannian submanifold $M$ of $\mathbb R^{r,s}$. Then the normal space decomposes as 
		$\nu _qM = \mathbb V_0 \oplus \cdots \oplus \mathbb V_k$, 
		orthogonal direct sum of non-degenerate $\Phi(q)$-invariant subspaces
		and $\Phi = \Phi _ 0\times \cdots \times \Phi_k$,  where $\Phi _0 = \{Id\}$ and 
		$\Phi_i $ acts trivially on $\mathbb V_j$ if $i\neq j$ and weakly irreducible on $\mathbb V_i$ for $i\geq 1$. (Note that \( \mathbb{V}_0 \) is contained, possibly properly, in the set of fixed vectors of \( \Phi(q) \).)
	\end{prop}
Note that the existence of a
splitting into weakly irreducible invariant subspaces of $\Phi(q)$,  plus a factor contained in the fixed vectors, is a standard fact. 
\vspace{.2cm}

	With the same proof of the normal holonomy theorem in 
	\cite{O1} (see also \cite {BCO}, Theorem 3.2.1) one obtains: 
	
	\begin{theorem}\label{th:NHT} 
		Let $M^{n,s}$ be an essentially Riemannian submanifold of $\mathbb R^{r,s}$ of the same signature as the ambient space. Then the restricted  normal holonomy $\Phi(q)$  of $M$ at $q$ acts on the orthogonal complement of its fixed set as the isotropy representation of a semisimple Riemannian symmetric space.
	\end{theorem}

	\subsection {Weakly polar actions}\label{sub:w-p-a}\ 
	
	Let $G$ act by isometries on a pseudo-Riemannian manifold 
	$M^{r,s}$, and let $\mathfrak g$ be its Lie algebra. Let $\Omega$ be the open and dense subset of $M$ such that the dimension of the $G$-orbits is locally constant. Let  
	$\mathcal V$ be the distribution on $\Omega$ given by the tangent spaces to the $G$-orbits, and let $\mathcal H := \mathcal V^\perp$
	be the distribution of normal spaces to the $G$-orbits. If $q\in \Omega$,  then $\dim \mathcal V _q + \dim \mathcal H _q = \dim M = r+s$. However  $\mathcal V_q\cap \mathcal H_q$ could be non-trivial if $\mathcal V_q$ is a degenerate subspace.

 \vspace{.2cm}
 
	The proof of the following lemma is standard. 
	
	\begin{lemma}\label{lem:wslice} Let $G$ be a Lie group acting on a manifold $M$ and let  $G\cdot p$ be a locally maximal dimensional orbit. Then 
 \begin{equation}\mathfrak g _p \, . \, T_pM\subset 
			T_p(G\cdot p) \end{equation}
        where $\mathfrak g_p$ denotes the isotropy algebra at $p$.
  
		%Let $q\in \Omega$ and let $\mathfrak g _q$ be the %isotropy algebra (regarded, via the isotropy %representation,  as a Lie subalgebra of 
  %$\mathfrak{so}%(T_qM)$). 
  %Then $\mathfrak g _q\mathcal H_q\subset 
		%\mathcal H_q\cap \mathcal V_q$. 
	\end{lemma}

\qed
 
	\begin{lemma}\label{lem:slice}
		Let $G$ be a Lie group of  isometries of a pseudo-Riemannian manifold $(M, \langle, \, ,\, \rangle)$. Let $G\cdot p$ be a (locally)  maximal dimensional, possible degenerate, orbit.
		Then the identity  component $G_p^o$ of the isotropy group at $p$ acts trivially on the normal space $\nu _p(G\cdot p)$.
	\end{lemma}

 \begin{proof}
		Let $\mathfrak g$ and $\mathfrak g _p$ be the Lie algebras of $G$ and $G_p$, respectively.
		Then, by Lemma \ref{lem:wslice}, 
		$$0 = 
		\langle \mathfrak g _p .T_pM, \nu _p(G\cdot p)\rangle 
		= \langle T_pM \, , \,  \mathfrak g_p \, . \, \nu _p(G\cdot p)\rangle $$
	\end{proof}

	The following lemma is well-known in the Riemannian case. The same arguments apply to the pseudo-Riemannian case.

	\begin{lemma}\label{lem:intg} We are under the previous notation and assumptions. 
		The distribution $\mathcal H$ is integrable if and only if it is autoparallel.
	\end{lemma}
	\begin{proof}
		Let $\xi, \eta$ be local fields on $\Omega$ that lie in 
		$\mathcal H$ and let $X$ be an arbitrary Killing field induced by $G$. 
		
		Since $\langle \xi , X\rangle = 0$, then 
		$0 = \eta \langle \xi , X\rangle = 
		\langle \nabla _\eta \xi , X\rangle + 
		\langle \xi , \nabla _\eta X\rangle $ (and the same is true by interchanging $\xi$ and  $\eta$). Then  
		$\langle \nabla _\eta \xi , X\rangle = -
		\langle \xi , \nabla _\eta X\rangle = 	
		\langle \eta , \nabla _\xi X\rangle$, 
		where the last equality is due to the Killing equation. 
		This implies that  $\langle \nabla _\eta \xi , X\rangle$ is skew-symmetric in $\xi , \eta$ and hence, 
		$\langle [\xi , \eta], X\rangle = 
		\langle \nabla _\xi  \eta - \nabla _\eta \xi,  X\rangle = 
		2\langle \nabla _\xi  \eta , X\rangle$.
	\end{proof}

	\begin{definition}\label{wP} The group $G$ acts {\it weakly polarly} on $M$ if the distribution $\mathcal H$ of $\Omega$ is integrable. 
	\end{definition}

	\begin{prop} \label{prop:Po} Let $\mathbb V$ be a vector space with a non-degenerate inner product of signature $s$, and  let  $G$ be a Lie subgroup of $SO(\mathbb V)$.  Assume that there exists an algebraic pseudo-Riemannian curvature tensor $R$ of $\mathbb V$ such that the curvature endomorphisms linearly span the Lie algebra  $\mathfrak g$ of $G$.  Furthermore, assume that $R$, regarded as a symmetric endomorphism of 
		$\Lambda ^2(\mathbb V)$, is positive semi-definite (i.e., if $u\in \Lambda ^2(\mathbb V)$ satisfies  $\langle R(u), u\rangle = 0$, then $R(u)=0$). Let $N$ be a non-degenerate submanifold of $\mathbb V$ which is locally invariant under the action of $G$. 
		\begin{enumerate}[(i)]
			\item $G$ acts weakly polarly on $N$.
			\item $T_q(G\cdot q)$ is invariant under any shape operator of $N$ at $q$, for all $q\in\Omega$, where $\Omega$ is the open and dense subset of $N$ where the dimensions of the $G$ orbits are locally constant.
		\end{enumerate}
		
	\end{prop}
	
	\begin{proof}
		Let $u,v\in \nu _qN$ and consider the Killing field $X$ of $\mathbb V$ given by $X_x = R_{u,v}x$. Then $X$ is a linear Killing field, so $\nabla _wX = R_{u,v}w$, where $\nabla$ is the usual   Levi-Civita flat connection of $\mathbb V$. 
		
		Let $\mathcal V$ be the distribution of $\Omega$ tangent to the $G$-orbits, and let 
		$\mathcal H = \mathcal V ^\perp$.
		If $x\in N$, then $\xi \in \mathcal H _x$ if and only if $0= \langle R_{u,v} x,\xi\rangle = \langle R_{x,\xi} u,v\rangle$ for all $u,v\in \mathbb V$.  Thus, $\mathcal H_x = 
		\{ \xi \in T_xN:  R_{x,\xi}=0 \}$. Let $\xi , \eta \in \mathcal H_x$. By making use of the Bianchi identity we have that 
		$R_{\xi , \eta}x=0$ and thus, $R_{\xi , \eta} \in \mathfrak g_x$.
		Then, from Lemma \ref{lem:wslice}, 
		$\langle R_{\xi , \eta}\mathcal  H_x,\mathcal H _x\rangle= \{0\}$, and therefore  $\langle R_{\xi , \eta}\xi , \eta\rangle = 0$. Since $R$ is positive semi-definite, we conclude that $R_{\xi , \eta}= 0$, for all $\xi , \eta \in \mathcal H_x$.
		
		Let $\tilde \xi, \tilde \eta$ be fields of $N$ that lie in 
		$\mathcal H$, and let $X$ be the field of $N$ given by 
		$X_x= R_{u,v}x$, where $u,v \in \mathbb V$ are arbitrary. Then $\langle \tilde \xi , X\rangle = 0$,  and hence, differentiating in the direction of $\eta$ one obtains 
		$$\langle \nabla _{\tilde \eta} \tilde \xi, X\rangle = - \langle \tilde \xi , 
		\nabla _{\tilde \eta}X\rangle = - \langle \tilde \xi , 
		R_{u,v} \tilde \eta\rangle = \langle R_{\tilde \xi , \tilde \eta} u,v\rangle =0$$
		and hence, since $u,v$ are arbitrary, $\mathcal H$ is autoparallel. This proves (i).
		
		Let $q\in \Omega$ and let us consider the orbit $G\cdot q$. Let 
		$\xi \in \mathcal H _q$ and let $\eta \in \nu _qN$ be arbitrary. Note that 
		$\xi, \eta$ are orthogonal to $T_q(G\cdot q) = \mathcal V _q$. Then, as in the proof of part (i), $R_{\xi , \eta}q = 0$, and so $R_{\eta ,\xi}$ belongs to the isotropy algebra $\mathfrak g_q$. Then, by Lemma 
		\ref{lem:wslice}, $R_{\eta, \xi }\xi$ belongs to $\mathcal V _q$. Thus,  $\langle R_{\eta , \xi}\xi, \eta \rangle =0$, which implies that  
		$R_{\eta , \xi}=0$. Let , for $u,v\in \mathbb V$, 
		$$\phi _t : = Exp (tR_{u,v})= \mathrm {e}^ {tR_{u,v}}.$$
		Then $\phi _t \xi$ is a field along $c(t): =\phi _t q$ that lies in $\mathcal H _{c(t)}$. Differentiating at $t=0$ one obtains 
		$R_{u,v}\xi = \frac{\mathrm D \,}{\mathrm dt}{\vert _0} \,\phi _t \xi$,  where  $\frac{\mathrm D \,}{\mathrm dt}$ is the ambient covariant derivative along the curve 
		$c(t)$.  Denote the second fundamental form and the shape operator of $N$ as $\alpha$ and $A$, respectively. Then 
		\begin{align*} 
			\begin{split}
			0&= \langle R_{\xi,\eta}u , v \rangle = \langle R_{u,v}\xi , \eta \rangle = \langle \alpha (c'(0), \xi) , \eta \rangle \\
			&= \langle A_\eta  c'(0), \xi \rangle =  \langle A_\eta  R_{u,v}q, \xi \rangle.
			\end{split}
			\end{align*}
		Since the vectors $R_{u,v}q$, $u,v\in \mathbb V$
		span $T_q(G\cdot q) = \mathcal V _q$ and $\xi \in \mathcal H _q$ is arbitrary, we conclude that  
		$A_\eta (T_q(G\cdot q))\subset T_q(G\cdot q)$
		
	\end{proof}

	\begin{theorem}\label{thm:NholPo} Let  $\Phi(q)$ be the  restricted normal holonomy group at $q$ of an essentially Riemannian submanifold $M$ of $\mathbb R^{r,s}$.  Let $N$ be a non-degenerate submanifold of the normal space $\nu _qM$ which is locally invariant by $\Phi(q)$.  Then $\Phi (q)$ acts weakly polarly on $N$. 
	\end{theorem}
	\begin{proof}
		The proof follows immediately from Lemma \ref{lem:unR} and Proposition  \ref{prop:Po}
	\end{proof}

 \begin{rem}\label{rem:parallel-shape} In a degenerate submanifold $S$ of a pseudo-Riemannian manifold,  the normal connection is not defined. Nevertheless,  a section $\tilde \xi$ of the normal bundle $\nu S = 
(T_pS)^\perp$ is called a {\it parallel normal field}, if  for any  tangent field $X$ of $S$, $\nabla _X\tilde \xi$ is a tangent field of $S$, where $\nabla$ is the Levi-Civita connection of the ambient space. Thus, the shape operator $A_{\tilde \xi}$ is defined by 
$A_{\tilde \xi}X := -\nabla _X\tilde \xi$. The same proof of the Gauss formula, since $\nabla$ is torsion-free and the bracket between tangent fields of $S$ is tangent to $S$, proves that $\langle A_{\tilde \xi}X,Y\rangle = \langle A_{\tilde \xi}Y,X\rangle$. Now, assume that $X_p$ is a degenerate vector, and let $Y_p$ be arbitrary. Then the above equality shows that $A_{\tilde \xi_p}X_p$ is degenerate. Then  $A_{\tilde \xi_p}$ leaves invariant the degeneracy subspace of 
$T_pS$. 
Analogously, a normal field $\eta (t)$ of $S$ along
a curve $c(t)$ is called {\it parallel} if $\frac {\mathrm{\, d}}{dt}\eta (t) \in T_{c(t)}S$. 
\end{rem}

\begin{cor}\label{cor:parallel-normal}  Let $G$ be a Lie group of  isometries of a pseudo-Riemannian manifold 
$(M, \langle, \, ,\, \rangle)$ which acts weakly polarly on $M$. Let $N=G\cdot p$ be a maximal dimensional orbit, with a (possible) degenerate induced metric. Then any $\xi \in \nu _pN$ extends, in a neighborhood $U$ of $p$ in $N$, to a parallel normal field $\tilde \xi$. 
\end{cor}
\begin{proof}   From Lemma \ref{lem:slice} it follows that $\xi$ extends to a $G$-invariant section $\tilde \xi$ of $\nu N$ in a neighbourhood $U$ of $N$. Since the arguments are local, we may assume that $U=N$ is an embedded submanifold of $M$. Let $\mathcal H$ be the autoparallel distribution given by the normal spaces to the $G$-orbits, defined in a neighbourhood $\Omega$ of $p$ in $M$ (see Lemma \ref{lem:intg}). Without loss of generality we may assume that $N\subset \Omega$. Since $\tilde \xi$ is tangent to $\mathcal H$, this normal field extends to a  field of $\Omega$ that lies in $\mathcal H$ (perhaps by making $\Omega$ smaller). We denote such an extension also by $\tilde \xi$. 
Let $X$ be a Killing field of $M$ induced by $G$, and let $\phi _t$ be its associated flow. Since the normal field $\tilde \xi$ is $G$-invariant, then 
$\mathrm{d}\phi_t \, (\tilde \xi _p) = \tilde \xi _{\phi _t(p)}$. Hence, 
\begin{equation}\label{eq:bra2}[X,\tilde \xi]_p=0
\end{equation}
	
	 Let $\tilde \eta$ be a field of 
$\Omega$ that lies in $\mathcal H$. Since $X$ is tangent to the $G$-orbits, then 
$\langle X, \tilde \eta\rangle =0$. By differentiating this equality in the direction of $\tilde \xi$ we obtain that 
$$ \tilde \xi _p  \langle X , \tilde \eta \rangle  = \langle  X_p , \nabla _{\tilde \xi _p}\tilde \eta  \rangle + 
\langle  \nabla _{\tilde \xi _p}X, \tilde \eta _p \rangle = 0. $$
Since  $\mathcal H$ is autoparallel, and by making use of (\ref{eq:bra2}), we obtain that $\langle  \nabla _{X_p}\tilde \xi , \tilde \eta _p\rangle =0$. Since $X$ and $\tilde \eta$ are arbitrary, we conclude that $\nabla _{T_pN}\tilde \xi \subset T_pN$.
\end{proof}

The proof of the following result is standard. In the case of a non-degenerate submanifold it is a special case of  Ricci identity.

\begin{lemma}\label{lem:comm} Let $M$ be a possible degenerate submanifold of $\mathbb R^{r,s}$ and let $\xi, \eta$ be parallel normal fields of $M$ (see Remark \ref{rem:parallel-shape}). Then 
	$\displaystyle{\langle [A_\xi, A_\eta]X,Y\rangle =0}$, 
	for all fields $X,Y$ tangent to $M$. 
	\end{lemma}
\qed

\subsection{Holonomy tubes around a focal manifold} \label{subs:2.3aaa}

Let $M\subset \mathbb R^{r,s}$ be a local submanifold 
with a non-degenerate induced metric. Let $\tilde \xi$ be a parallel normal  field of $M$ and assume that $0<\dim \ker (Id -A_{\tilde \xi (x)})<\dim M$,  and that $\dim \ker (Id -A_{\tilde \xi (x)})$ is independent of $x\in M$, where $A$ is the shape operator of $M$.

\vspace{.15cm}

\underline{\it Assumptions}: The {\it vertical} distribution $\ker (Id -A_{\tilde \xi (x)})$ of $M$ is  pseudo-Riemannian and the {\it horizontal}  distribution $\mathcal H ^{\tilde \xi}:=(\ker (Id -A_{\tilde \xi (x)}))^\perp$ is Riemannian.

\vspace{.15cm}

By the Codazzi equation, 
$\ker(Id - A_{\tilde \xi (x)})$ defines an autoparallel distribution of $M$. Let us consider, locally,  the Riemannian parallel focal manifold $M_{\tilde{\xi}}=\{ x + \tilde \xi (x):x\in M\}$. If $\pi : M\to M_{\tilde \xi}$ is the projection, i.e. $\pi (x) = x +\tilde \xi (x)$, then 
$\ker {\mathbf {d}\pi} = \ker(Id - A_{\tilde {\xi} })$. Observe that $T_{\pi (x)}M_{\tilde \xi} =  (\ker {\mathbf {d}_x\pi})^\perp \subset T_xM$, as subspaces of the ambient space.
Moreover, any fiber $\pi^{-1}(\{\pi (x)\})$ is contained in the (affine) normal space $\pi (x) +\nu _{\pi (x)}M_{\tilde \xi}$, and   the $\nabla ^\perp$-parallel transport $\tau ^\perp _c$ along an arbitrary curve $c$ of $M_{\tilde \xi}$ from  $\pi (x)$ to $\pi (y)$ maps (locally) $\pi^{-1}(\{\pi (x)\})$ into $\pi^{-1}(\{\pi (y)\})$ (see \cite[Lemma 3.4.10]{BCO}).
 In particular, 
\begin{equation}\label{eq:832}
\pi (x) + \Phi (\pi (x))\cdot (x-\pi (x))\subset \pi^{-1}(\{\pi (x)\}) \text {\ \ \ (locally),} \end{equation}
 where $\Phi$ denotes the local normal holonomy group of $M_{\tilde \xi}$. We regard, in the obvious way, this parallel transport as a map from the affine normal spaces, i.e.  
 $\tau ^\perp _c: \pi (x) + \nu _{\pi (x)} M_{\tilde \xi}\to \pi (y) + \nu _{\pi (y)} M_{\tilde \xi}\subset \mathbb R^{r,s}$. If $v\in T_x\, \pi^{-1}(\{\pi (x)\})$, then $\mathrm{d}\tau ^\perp _c(v)$ is naturally identified with the linear parallel transport $\tau ^\perp _c(v)$. Any of these  possible interpretations of the normal parallel transport will be clear from the context.  

 One has that $M$ is (locally) foliated by the holonomy tubes (see \cite[p. 220]{BCO}) 
\begin{equation}\label{349}
H^{\tilde \xi}(x): = (M_{\tilde \xi})_{x-\pi (x)} = 
(M_{\tilde \xi})_{-\tilde \xi (x)}
\end{equation}
 By considering a smaller neighborhood of a nearby generic point, we may assume that all holonomy tubes have the same dimension, or equivalently,  that 
$ \dim (\Phi (\pi (x))\cdot (-\tilde \xi (x))$ does not depend on $x \in M$.  
The induced metric on $(M_{\tilde \xi})_{x-\pi (x)}$ may be degenerate at $x$. This occurs if and only if $\Phi (\pi (x))\cdot (-\tilde \xi (x))$ is a degenerate orbit. Additionally, we may assume that the dimension of the degeneracy of the induced metric on $H^ {\tilde \xi}(x)$ is constant. 

Let $\tilde \nu$ be the distribution of $M$ perpendicular  to the distribution $\mathcal T$ defined  by the tangent spaces of the holonomy tubes. 
When the holonomy tubes are degenerate, then 
$\tilde \nu$ and $\mathcal T$ have a non-trivial intersection. Let us consider the  distribution $\mathcal H^{\tilde \xi} = (\ker \mathbf {d}\pi )^\perp $ and 
observe that $(\mathcal H^{\tilde \xi}) _x = T_{\pi (x)}M_{\tilde \xi}$ (as linear subspaces of the ambient space). Moreover, 
\begin{equation}\label{eq:tanhol}
	\mathcal T _x := T_xH^{\tilde \xi}(x) = T_x\biggl(\pi (x) + \Phi (\pi (x))\cdot \bigl(x-\pi (x)\bigr)\biggr)\oplus (\mathcal H^{\tilde \xi}) _x
\end{equation}
Since $\tilde \xi$ is a parallel normal field of $M$,  by the Ricci identity,  the shape operator $A_{\tilde \xi}$ commutes with any other shape operator of $M$. Thus, $\ker \mathbf {d}\pi $ and $\mathcal H$   are distributions which are invariant under all shape operators of $M$. 
From the Codazzi identity, it follows that  the distribution $\ker \mathbf {d}\pi$ is autoparallel. Furthermore, from the construction of the holonomy tubes inside $M$, and by making use of  Theorem \ref{thm:NholPo}, the distribution $\tilde \nu$ is  autoparallel, and contained in $\ker \mathbf {d}\pi$.

Observe that the normal space $\nu _xM$ of $M$ at $x$ coincides with the normal space at $x$ of $\pi^{-1}(\{\pi (x)\})$, regarded as a submanifold of the affine normal space $\pi (x) +\nu _{\pi (x)}M_{\tilde \xi}$. Then, taking into account that $\ker \mathbf {d}\pi$ is invariant under all the shape operators of $M$  and Proposition  \ref{prop:Po}, we obtain the following results (keeping the assumptions and notation of this section).

\begin{lemma}\label{lem:743}  The distributions $\ker \mathbf {d}\pi $, $\mathcal H^{\tilde \xi}$, $\mathcal T$, and  $\tilde \nu$ are invariant under all shape operators of $M$. Moreover,  $\ker \mathbf {d}\pi$ and $\tilde \nu$ are autoparallel. \qed
\end{lemma}

\begin{cor}\label{cor:743}  Let $\tilde{\eta}$ be a parallel normal field of $M$. Then  $\tilde{\eta}_{\vert H^{\tilde \xi}(x)}$ is a parallel normal field of $H^{\tilde \xi}(x)$,  for all $x\in M$. In particular,  $\tilde{\xi}_{\vert H^{\tilde \xi}(x)}$ is a parallel normal field of $H^{\tilde \xi}(x)$. 
	\qed 
\end{cor}
\noindent (The definition of a parallel normal field, if $H^{\tilde \xi}$ is degenerate, is given by  Remark \ref{rem:parallel-shape}).

\vspace{.15cm}

Let  $c(t)$ be a  horizontal curve in 
$H^{\tilde \xi}(x)$ and let $\eta (t)$ be a normal filed of $H^{\tilde \xi}(x)$ along $c(t)$. Then  it is standard to show, and well-known in a Euclidean ambient space by an argument that goes back to \cite{HOT}, 
that $\eta (t)$ is  a parallel normal field along $c(t)$ if and only if $\eta (t)$ is a parallel normal field of 
 $M_{\tilde \xi}$ along the curve 
$\pi (c(t))$. 

\begin{rem}\label{rem:dsjk}
    The distribution of $M$ tangent to the normal holonomy orbits of the focal manifold is given by $ \mathcal T \cap \ker \mathbf {d}\pi$; see
    (\ref{eq:tanhol}).
\end{rem}

\begin{lemma}\label{lem:744}  Let $\psi\in \tilde{\nu}_x$. Then  $\psi$ extends (locally) to a section $\tilde \psi$ of 
$\tilde \nu _{\vert H^{\tilde \xi}(x)}$, which is a parallel normal field of $H^{\tilde \xi}(x)$. Moreover, the shape operator 
$\hat A_ {\tilde \psi}$  of  $H^{\tilde \xi}(x)$ leaves invariant the horizontal distribution  $\mathcal H^{\tilde \xi}$.
\end{lemma}
\begin{proof}
	 If $H^{\tilde \xi}(x)$ is non-degenerate,  the proof follows analogous arguments to those in part (iii) of Proposition 7.1.1 of \cite{BCO}. In the degenerate case the arguments are similar, after applying  Corollary \ref{cor:parallel-normal} to construct a parallel normal field of the degenerate  normal holonomy orbit $\pi (x) + \Phi (\pi (x))\cdot (x-\pi (x))$ of the focal manifold $M_{\tilde \xi}$. In fact, let  $y\in H^{\tilde \xi}(x)$, and let $c: [0,1] \to H^{\tilde \xi}(x)$ be a horizontal curve with $c(0) =x$, $c(1)= y$. Let $\bar{\psi} (t)$ be the  parallel normal field along the curve $\pi (c(t))$ of the Riemannian manifold $M_{\tilde \xi}$ with $\bar{\psi} (0)= \psi$. Then $\bar{\psi} (t)$ is a parallel normal field of $H^{\tilde \xi}(x)$ along the curve $c(t)$ (see Remark \ref{rem:parallel-shape}). We define $\tilde \psi (y) = \bar{\psi} (1)$. From 
	Corollary \ref{cor:parallel-normal} one obtains that $\tilde \psi$ is well defined (near $x$), defines a parallel normal field of 
	$H^{\tilde \xi}(x)$. The last assertion follows from the construction of $\tilde \psi$
\end{proof}

\begin{rem}\label{rem:eq-class} Let us define on $M$ the following equivalence class:
	 $x\underset{\tilde \xi}{\sim}y$ if there is a curve
	in $M$ from $x$ to $y$ which lies in the horizontal distribution $\mathcal H$. Let $[x]$ denote the equivalence class of $x$. Then, locally, 
	$$H^{\tilde \xi}(x) = [x]$$
	(see last paragraph of \cite[p. 224]{BCO}).
\end{rem}

%{
	\subsection{The horosphere embedding} \label{sec:horosphere}

	Let $\mathbb R ^{r,s}$ be the pseudo-Euclidean space $\mathbb R ^{r+s}$ with signature $s$ where the inner product is given by $\langle v, v\rangle = -v_1^2 - \cdots -v_s^2 + v_{s+1}^2 + \cdots  + v_{s+r}^2$. 
	
	The {\it horosphere  embedding} is the isometric map $f:\mathbb R ^{r,s}\to \mathbb R ^{r+1,s+1} \simeq \mathbb R ^{1,1}\times \mathbb R ^{r,s}$ given by 
	%cambio f por h
	\begin{equation}\label{eq:def}  
		f(x) = (\frac 12 \langle x,x\rangle +1, \frac 12 \langle x,x\rangle, x)
	\end{equation}
	Then $Q^{r,s} : = f(\mathbb R ^{r,s})$ is called  the {\it pseudo-horosphere} of the pseudo-hyperbolic space 
	\begin{equation} \label{eq:pHip} H^ {r+1,s} = \{ v \in \mathbb R ^{r+1,s+1}: \langle v,v\rangle=-1 \}^o
	\end{equation}
	where $e_{-1}, e_0, \cdots , e_{r+s}$ is the canonical basis of $\mathbb R ^{1,1}\times \mathbb R ^{r,s}$ and
	$\{\, \}^o$ denotes the connected component by $e_{-1}$ (we will frequently write $Q$ instead of $Q^{r,s}$).  Namely, 
	\begin{equation}\label{eq:horo}
		Q= H^ {r+1,s}\cap E
	\end{equation}
	where $E$ is the degenerate affine subspace of $\mathbb R ^{1,1}\times \mathbb R ^{r,s}$ given by the equation $x_{-1} - x_0 = 1$. One has that $f: \mathbb R ^{r,s}\to 
	Q$ is an isometric diffeomorphism and the map $f: \mathbb R ^{r,s}\to \mathbb R ^{r+1,s+1}$ is an isometric $\rho$-equivariant embedding, is a Lie group morphism from the isometry group of  $\mathbb R ^{r,s}$ into the orthogonal group $O(r+1,s+1)$.
 In fact, let $g\in O(r,s)$ 
	 and let $\tau _v$ be the translation by $v\in \mathbb R ^{r,s}$ in $\mathbb R ^{r,s}$. Then $\rho (g)$ is given by  the natural   inclusion of $O(r,s)\subset O(r+1,s+1)$, where 
	 $\mathbb R ^{r,s}\subset \mathbb R ^{1,1} \times \mathbb R ^{r,s}=
	\mathbb R ^{r+1,s+1}$. Moreover,  
	\begin{align} \label{eq:equiv} 
		&\rho (\tau _v)(x_{-1}, x_0, x) = \\ \nonumber &
		(x_{-1} +\langle x, v\rangle + \frac 12 (x_{-1} -x_0)\langle v, v\rangle, 
		x_{0} +\langle x, v\rangle + \frac 12 (x_{-1} -x_0)\langle v,v\rangle, 
		x + (x_{-1} -x_0)v)
	\end{align}

	One has that $Q^{r,s}\simeq \mathbb R^{r,s}$ is a pseudo-Riemannian flat manifold of signature $s$. 
	
	If $c(t)= (c_{-1}(t), c_0(t), \cdots , c_{r+s}(t))$ is a curve in 
	$Q$, then $c_{-1}(t)- c_0(t)=1$ and hence, differentiating, 
	$0 = c_{-1}'(t) - c_{0}'(t) = \langle -e_{-1} + e_0, c'(t)\rangle  $. Then $\xi ^0= -e_{-1} + e_0$ is a constant $\nabla ^\perp$-parallel  normal vector field to $Q$. Moreover, if $A$ is the shape operator of $Q\hookrightarrow \mathbb R ^{r+1,s+1}$, then $A_{\xi ^0 }= 0$.
	
	The position vector field $\xi ^1$ of $H^ {r+1,s} \subset \mathbb R^{r+1,s+1}$ is an umbilical parallel normal field. Namely, 
	$A'_{\xi ^1}= - \mathrm{Id}$, where $A'$ is the shape operator of 
	$H^ {r+1,s}$. Thus, the restriction of $\xi ^1$ to $Q$ is also a parallel normal field and $A_{\xi ^1}= - \mathrm{Id}$, where $A$ is the shape operator of the horosphere. 
	Then the  normal space $\nu Q$ of $Q$ in $\mathbb R ^{r+1,s+1}$ is generated by the parallel independent normal fields $\xi^0,\xi^1$, which are umbilical. Let  $i:M\to Q\simeq \mathbb R^{r,s}$ be an isometric immersion and let  $\nu M$ be  the normal bundle of $M$. Then the normal bundle of $M$, regarded as a submanifold of $\mathbb R^{r+1,s+1}$,  decomposes orthogonally as 
	\begin{equation}\label{eq:1}
		\bar \nu M = i^*(\nu Q) \oplus \nu M
	\end{equation}
	where $i^*(\nu Q)$ is the pull-back bundle, which is a parallel, flat and umbilical sub-bundle of $\bar \nu M$.
	
	\vspace{.15cm}
	
	\begin{rem}
Since the pseudo-horosphere \( Q^{r,s} \) is umbilical, it follows that, via the horosphere embedding, it is an essentially Riemannian submanifold of \( \mathbb{R}^{r+1,s+1} \).
\end{rem}

The proof of the following lemma is the same as that for Euclidean submanifolds when dealing with the zero distribution associated to  the kernel of the shape operator of a parallel normal field (see \cite[Section 7.1]{BCO}).
\begin{lemma}\label{lem:horo643} Let $M$ be a local  pseudo-Riemannian submanifold and let $\tilde \eta$ be a parallel normal field such that the kernel of the shape operator $A_{\tilde \eta}$ has constant dimension. We identify, by means of the horosphere embedding, $M$ with its image $\tilde M$ under the horosphere embedding and $\tilde \eta$ with a parallel normal  field of $\tilde M$ (tangent to the horosphere). Let $\tilde v$ be the position (parallel normal) field of $\tilde M$. Then 
	$\ker A_{\tilde{\eta}} = \ker (Id - \tilde A_{\tilde \eta -\tilde v})$, where $\tilde A$ is the shape operator of $\tilde M$. Thus, $\ker A_{\tilde{\eta}}$  is the vertical distribution associated to the projection $pr: \tilde M \to 
	\tilde M_{\tilde \eta -\tilde v}$, $pr(x)= x + \tilde{\eta} (x) -\tilde v(x) = \tilde{\eta} (x)$.
\end{lemma}

\section {The lift of a K\"ahler submanifold of  $\mathbb CH^n$ to $\mathbb C^{n,1}$}  

	Let $\mathbb C^{n,1}$ be the complex space $\mathbb C^{n+1}$ endowed with the pseudo-Hermitian inner product $\langle \, , \,\rangle ^H$ given by 
	\begin{equation} \langle (z_0,z_1, \cdots , z_n), (z'_0,z'_1, \cdots , z'_n)\rangle ^H
		= -z_0\bar z'_0 + z_1\bar z'_1 + \cdots + z_n\bar z'_n.
	\end{equation}
	The induced (real) inner product, i.e.\, the real part of the pseudo-Hermitian inner product will be denoted by 
	$\langle \, , \,\rangle$.
	Let $\langle\langle \, , \, \rangle\rangle ^H$ be the canonical Hermitian inner product of $\mathbb C^{n+1}$. This Hermitian inner product induces the canonical inner product of $\mathbb C^{n+1} \simeq \mathbb R^{2n+2}$ which will be denoted by   $\langle\langle \, , \, \rangle\rangle$. 
	
	Observe that 
	$\langle \, ,\, \rangle$ and $\langle \langle \, ,\, \rangle\rangle$ naturally induce on $\mathbb C^{n+1}$ flat pseudo-Riemannian and  Riemannian metrics, respectively. Such metric tensors will be also denoted 
	$\langle \, ,\, \rangle$ and $\langle \langle \, ,\, \rangle\rangle$, respectively. 
	Nevertheless, the associated Levi-Civita connections coincide. In fact, it is the usual connection $\nabla$ of a vector space. 
	The K\"ahler structure $J$ of $\mathbb C^{n+1}$ is also a pseudo-K\"ahler structure of $\mathbb C^{n,1}$

	Observe that $\mathbb C^{n,1}$, regarded as a real pseudo-Euclidean space, has signature $2$ and thus $\mathbb C^{n,1}\simeq \mathbb R^{2n,2}$. 
	Let 
	\begin{equation}\label{eq:neg} \mathbb C^{n,1}_{-} =\{z\in \mathbb C^{n,1}:
		\langle z,z\rangle <0\}
	\end{equation}
	which is an open subset of $\mathbb C^{n,1}\simeq \mathbb C^{n+1}$.
	Observe that  
	$\lambda \mathbb C^{n,1}_{-}= \mathbb C^{n,1}_{-}$, for any  $\lambda \in \mathbb C^*= \mathbb C-\{0\}$. 
	
	The  complex hyperbolic space  $\mathbb CH^ n$   is the projectivized space of $\mathbb C^{n,1}_{-}$. Moreover, it is the symmetric dual space of the complex projective space $\mathbb CP^n$.  The symmetric presentation is 
	$$\mathbb CH^ n = \mathrm{SU}_{n,1}/\mathrm {S}(\mathrm U_1\mathrm U_n),$$ where the group $\mathrm{SU}_{n,1}$ is the group of complex linear transformations of $\mathbb C^{n+1}$ that preserve $\langle \, , \,\rangle$. 
	The Riemannian metric on $\mathbb CH^ n = \mathrm{SU}_{n,1}/\mathrm {S}(\mathrm U_1\mathrm U_n)$  , up to a scaling, is unique and has constant and  negative holomorphic curvature. We choose such a Riemannian metric to have holomorphic curvature equal to  $-4$. 
	Observe that $\mathbb CH^ n$ may be regarded as an open subset of $\mathbb CP^n$; see (\ref{eq:neg}). But the symmetric Riemannian metric is different.
	
	Let $\pi :  \mathbb C^{n,1}_{-} \to \mathbb CH^ n$ be the projection.  
	Then $\pi$ is a submersion and 
	\begin{equation}\label{eq:dpi}
		\ker (\mathrm d\pi)_q= T_q(\mathbb C^*q)\simeq \mathbb Cq 
	\end{equation}
	\begin{defi}\label{def:lift} The lift $N$ of a submanifold $\bar N$ of $\mathbb CH^n$, to a submanifold of 
		$\mathbb C^{n,1}$, is $N = h(\pi ^{-1}(\bar N))$, where $h: \mathbb C^{n,1}_{-} \to  \mathbb C^{n,1}$ is the inclusion.
	\end{defi}
    
Such a lift, when restricted to the anti--de Sitter space, has been considered in \cite{DDS} with the aim of classifying isoparametric hypersurfaces in the complex projective space.

\vspace{.2cm}
    
	One has, from (\ref{eq:dpi}), that the lift of a submanifold of  $\mathbb CH^n$ is a non-degenerate submanifold of $\mathbb C^{n,1}$ with signature $2$. Moreover, $\pi: \mathbb C^{n,1}_- \to \mathbb CH^n$ is a fibration with fibers 
	$\pi ^{-1}(\{\pi (q)\})= \mathbb C^*q$,  
	where $\mathbb C^*= \mathbb C-\{0\}$. 
	Let  $\mathcal V$ be the vertical distribution of $\mathbb C^{n,1}_-$; i.e
	tangent to the fibers of $\pi$. One has that $\mathcal V_q = \mathbb Cq$, regarded as a subspace of $T_q\mathbb C^{n,1}$.
	Observe that, for all $q\in  \mathbb C^{n,1}_-$, 
	$\mathcal V_q$ is a negative $2$-dimensional (real) subspace of $T_q\mathbb{C}^{n,1}$. Then the perpendicular distribution, the so-called horizontal distribution,  $\mathcal H :=\mathcal V^\perp$ is a Riemannian distribution. 
	
	Let us consider the (real) pseudo-hyperbolic space 
	\begin{equation}\label{eq:pHip2}
		H^{2n,1}_r = \{v\in \mathbb C^{n,1}: \langle v , v\rangle =-r^ 2\}\subset \mathbb C^{n,1}_-
	\end{equation}
	of constant negative curvature $-1/r^2$, $r>0$; cf. (\ref{eq:pHip}). 
	Observe that for any $r>0$,   $\pi: 	H^{2n,1}_{r} \to \mathbb CH^n$ is a submersion. Moreover, it is a fibration with non-degenerate negative definite fibers  $S^1\cdot q$, where $S^1$ here denotes the unit complex  numbers. 
	The vertical distribution at $q$ is given by  $\mathcal V_q\cap T_q	H_r^{2n,1} = Jq$. The horizontal distribution is just the restriction of $\mathcal H$ to 
	$H_r^{2n,1}$. 
	It is well-known that 
	$\pi _{\vert H^{2n,1}_r}$ is a pseudo-Riemannian submersion of factor 
	$1/r$, i.e. 
	$\mathrm d _q \pi : \mathcal H_q  \to T_{\pi(q)}\mathbb CH^ n$ is a homothety of factor $1/r$. Namely, 
	$$\langle \mathrm d_q\pi (u), \mathrm d_q\pi (u)\rangle = r^{-2} 
	\langle u ,u\rangle .$$
	
	The distributions $\mathcal V$ and $\mathcal H$ of $\mathbb C^{n,1}_-$ are both $J$-invariant.  
	Moreover, if $\bar J$ is the K\"ahler structure of 
	$\mathbb CH^n$, one has that $\mathrm d\pi (Jv) = 
	\bar J  \mathrm d\pi (v)$, for all 
	$v\in T\mathbb C^{n,1}_-$. This implies that   $N=\pi^{-1}(\bar N)$ is a pseudo-K\"ahler submanifold of $\mathbb C^{n,1}$ if and only if $\bar N$ is a K\"ahler submanifold of $\mathbb CH^n$.
	
	Recall that a submanifold of a Riemannian manifold is called {\it full} if it is not contained in a proper totally geodesic submanifold of the ambient space. 
	
	\vspace{.15cm}

	 \begin{rem}\label{rem:full7} If  $X$ is  either $\mathbb CP^n$ or $\mathbb CH^n$, then any totally geodesic submanifold of $X$ is  complex or totally real. 
		Assume that  a Kähler submanifold $\bar N$  
 is  contained in a  totally geodesic  submanifold $\bar\Sigma$ of $X$.
		Then $\bar\Sigma$ is Kähler. 
	\end{rem}
	
	If $N$ is a submanifold of a real vector space $\mathbb V$ and $q\in N$, then the affine subspace generated by the set $N$ coincides with $q + \mathbb W$, where $\mathbb W$ is the (real) linear subspace generated by all the tangent spaces of $N$. If $\mathbb V$ is complex and $N$ is K\"ahler then any tangent space is complex and so $\mathbb W$ is complex.  
	Let now $N  \subset  \mathbb C^{n,1}_-\subset \mathbb V: =\mathbb C^{n,1}$ be the lift of a K\"ahler submanifold of $\mathbb CH^n$. Then, for any  given $q\in N$, $\mathbb C^*q \subset q + \mathbb W$ and so the limit point $0$ belongs to the affine subspace generated by $N$. Then $q + \mathbb W = \mathbb W$ and $\mathbb W$ is a complex subspace of $\mathbb C^{n+1}$.
	Since $\mathbb C q \subset T_qN$ and $\mathbb C q$ is a negative definite complex line, we obtain that 
	the signature of $\mathbb W$ is $2$.
	
	\vspace{.15cm}
	
	It is well-known, and standard to proof, that  $\bar \Sigma$ is a totally geodesic K\"ahler submanifold of $\mathbb C H^n$ if and only if its lift $\Sigma$ is the intersection of a complex subspace $\mathbb W$ of signature $2$ with $\mathbb{C}^{n,1}_-$. 
	Then, the previous discussion and  Remark \ref{rem:full7} imply:
	
	\begin{prop}\label{prop:fullif}
		A submanifold $\bar N$ of $\mathbb CH^n$ is full if and only if its lift $N$ is a full submanifold of $\mathbb C^{n,1}$.
	\end{prop}	
	One has the following result:
	 \begin{lemma}\label{lem:nullity43}
		Let $\bar N$ be a K\"ahler submanifold of $\mathbb CH^n$, and let 
		$N$ be its lift  to $\mathbb C^{n,1}$. Let  ${\mathcal N} _q$ be the nullity of the second fundamental form of $N$ at $q$, and let  $\bar {\mathcal N} _{\pi(q)}$ be the nullity of the second fundamental form of $\bar N$ at 
		$\pi (q)$. Then, for all $q\in N$, 
		\begin{enumerate}[(i)]
			\item $\mathcal V_q \subset {\mathcal N} _q$
			
			\item $ {\mathcal N} _q = (\mathrm d_q\pi)^{-1}(\bar{\mathcal N} _{\pi(q)})$.
			
		\end{enumerate}
		In particular, if $\bar{\mathcal N} _{\pi(q)} = \{0\}$, then 
		${\mathcal N} _q = \mathcal V_q$.
	\end{lemma}
	\begin{proof} Recall that the lift   of a K\"ahler submanifold of $\mathbb CH^n$ is a pseudo-K\"ahler submanifold of $\mathbb C^{n,1}$ and  
		observe that $\mathbb C^*N= N$. 
		Let, for $\lambda \in \mathbb C^*$, $\mu _ \lambda: \mathbb C ^{n,1}\to \mathbb C ^{n,1}$ denote the multiplication by  $\lambda$. If $q\in N$, then 
		$
		T_{\mu _\lambda (q)}N = \mathrm d \mu _\lambda (T_qN) = 
		\lambda (T_qN)= T_qN$. 
		This means that the tangent spaces of $N$ are constant along any fiber. This implies that 
		$\mathcal V _{\vert N}\subset  { \mathcal N}$. 
		This shows (i). 
		
		Let $-r^2 = 
		\langle q , q\rangle$,  let $\bar X,\bar Y$ be fields of $\bar N$ around $\pi (q)$ and let  $ X$, $Y$ be their horizontal lifts to $H^{2n,1}_r$.
		Since the normal space $\mathbb Rq$  of $H^{2n,1}_r$ at $q$ is included in ${\mathcal N} _q$, one obtains that $v\in T_qH^{2n,1}_r$ belongs to $\mathcal N _q$ if and only if $v$ is in the nullity of the second fundamental form $\hat \alpha$ of $\hat N := N \cap H^{2n,1}_r$ as a submanifold of 
		$ H^{2n,1}_r$.  If $\hat \nabla $ is the Levi-Civita connection of $H^{2n,1}_r$ we obtain, from O'Neill formulas that 
		$$\mathrm d \pi (\hat \nabla _{X_q} Y)= 
		\bar \nabla _{\bar X_q}\bar Y,$$ where $\bar \nabla$ is the Levi-Civita connection of $\mathbb CH^n$. Since   the normal space of $\hat N$ in $ H^{2n,1}_r$ is included in $\mathcal H _q$, we obtain, by taking normal components,  that 
		$$\mathrm d \pi\hat \alpha ( X_q,  Y_q)= \bar \alpha (\bar X_q,\bar Y_q), $$
		where $\bar \alpha$ is the second fundamental form of $\bar N$. From this it follows (ii).
	\end{proof}
\color{black}
	It is clear that the  normal holonomy of pseudo-K\"ahler submanifolds of pseudo-K\"ahler spaces 
	acts by complex endomorphisms.

	\begin{rem}\label{rem:Rli}
		Since the restriction of vertical distribution   $\mathcal V$ to $N$ is tangent to $N$, for any $q\in  N$, $\mathrm d_q\pi: \nu _qN\to \nu _{\pi(q)}\bar N$ is a homothecy of factor $r^{-2}$, where 
        $r^2=-\langle q,q\rangle$ . \end{rem}
	
	We have the following result: 
	
	\begin{lemma}\label{lem:NH-compare}
		Let $\bar N$ be a K\"ahler submanifold of $\mathbb CH^n$ and let 
		$N$ be its lift  to $\mathbb C^{n,1}$. Let $ q \in N$ be arbitrary and let $\bar q= \pi ( q)$. Then 
		$$S^1\bar \Phi (\bar q) = \mathrm d_q\pi (S^1 \Phi (q)):= \mathrm d_q\pi \circ ( S^1 \Phi (q)) \circ  \mathrm (d_q\pi_{\vert \nu _qN})^{-1}$$
		where $\Phi $ and $\bar{\Phi} $ are the local  normal holonomy groups of $N$ and $\bar N$, respectively and 
		$S^1$ is the group of unit complex numbers acting on the   normal spaces.
	\end{lemma}
	\begin{proof} The arguments are the same as those inside  the proof of Lemma 7.5.4 of \cite{BCO} for proving formula 
		(7.7) there. 
	\end{proof}
	
	\vspace{.15cm}
	
	 Let us recall that the  {\it index of relative nullity}, of a non-degenerate submanifold of a pseudo-Riemannian manifold, is the dimension of $\mathcal N _q$ where  $\mathcal N _q$ is the nullity subspace  of the second fundamental form at $q$. The set of points where the index of relative nullity attain its minimum  is open. If the submanifold is connected  and analytic, then this set is also dense. 
	
	\vspace{.15cm}
	
	We recall a result from Alekseevsky and Di Scala (see Theorem 1, Theorem 2 and Corollary 1 of \cite{AD}):
	
	\begin{theorem}[\cite{AD}]\label{th:AD} Let $\bar N$ be a K\"ahler submanifold of a space of constant holomorphic curvature. If the index of relative nullity at $\bar q$ is zero, then the restricted normal holonomy group $\bar \Phi (\bar q)$ acts on the normal space as the isotropy representation of an irreducible Hermitian symmetric space. In particular, $\bar \Phi (\bar q)$ contains the group of multiplications by unit complex numbers.
	\end{theorem}

\begin{rem}\label{rem:local-hol} If one replace in Theorem \ref{th:AD} the restricted normal holonomy group by the local holonomy group, then the conclusion is the same. In fact, the local normal holonomy group at $p$ is the normal holonomy group at $p$ of a small simply connected neighbourhood of $p$.
\end{rem}
	
	Observe that the lift of a K\"ahler submanifold of $\mathbb CH^n$ is an essentially Riemannian submanifold $\mathbb C^{n,1}$ with the same signature. Then Lemma \ref{lem:NH-compare}, Theorem 
	\ref{th:NHT}, and Theorem \ref{th:AD} imply: 
	
	\begin{cor}\label{cor:NH-compare}
 Let $\bar N$ be a K\"ahler submanifold of $\mathbb CH^n$ with index of relative nullity $\nu _{\bar q} = 0$ at $\bar q \in \bar N$. Let 
		$N$ be the lift of $\bar N$  to $\mathbb C^{n,1}$, let  $ q \in \tilde N$ be such that  $\pi (q) = \bar q$. Then 
		$$\bar \Phi (\bar q) = \mathrm d_q\pi ( \Phi ( q))$$
		where $\bar \Phi $ and $ \Phi $ are the  local  normal holonomy groups of $\bar N$ and $N$, respectively. Moreover, $\bar \Phi (\bar q)$  and ${\Phi}(q)$ act irreducibly as the isotropy representation of a Hermitian symmetric space.
	\end{cor}

\begin{lemma} \label{lem:int3}
Let $\bar N$ be a K\"ahler submanifold of 
$\mathbb CH^n$ and let 
$N$ be the lift of $\bar N$  to $\mathbb C_{-}^{n,1}$. Let
$\mathcal V '$ be the restriction to $N$ of the vertical distribution $\mathcal V$ of $\mathbb C_{-}^{n,1}$ and let 
$\mathcal H' = \mathcal V'^\perp= \mathcal H \cap TN$. 
Then $\mathcal H'$ has no integral manifolds.
\end{lemma}
\begin{proof}  Assume that $N'$ is an integral manifold of $\mathcal H'$. Since $\mathcal V '$ is $J$-invariant, then $N'$ is a pseudo-K\" ahler (Riemannian) submanifold of $\mathbb C_{-}^{n,1}$. 
	Observe, since $N'$ is always perpendicular to the position vector, that $N'\subset H^{2n,1}_r$, where 
	$-r^2 = \langle q, q\rangle$ is independent of $q\in N'$.
	Let $\xi$ be the restriction to $N'$ of the position vector field, which is an umbilical parallel normal vector field of $N'$, i.e. 
	$A_{\xi}= - Id$ where $A$ is the shape operator of $N'$.  Observe that $J\xi$, as well as 
 $\xi$, is a parallel normal filed  and $A_{J\xi}= JA_{\xi} = -JId$.  The left hand side of this equality is a symmetric $(1,1)$  tensor on $N'$ while the right hand side is skew-symmetric and non-null. A contradiction.
\end{proof}

%Let $\bar N$ be a K\"ahler submanifold of 
%$\mathbb CH^n$ with index of relative 
%nullity equals to zero. %Let 
%$N$ be the lift of $\bar N$  to 
%$\mathbb C_{-}^{n,1}$. Then, by 
%Corollary \ref{cor:NH-compare} the restricted 

%\begin{theorem}	
%\end{theorem}

	\section{Holonomy tubes and the canonical foliation}\label{sec:NT1}

	The general arguments for this section are be based on 
	\cite{CDO}, \cite[Section 7]{BCO}, but our notation is slightly different for the restricted normal holonomy groups. We will adapt the arguments in these references to the pseudo-Riemannian case. Moreover, we will simplify  some crucial proofs there.   The main difficulty is to deal with degenerate orbits of  normal holonomy groups associated to focalization at infinity of the leaves of  nullity foliations. 
	
	We keep the general notation of previous sections.
	
	\vspace{.15cm}
	
	\underline{\it General assumption}: $\bar N$ is a K\" ahler (local) submanifold of $\mathbb CH^ n$ with zero index of relative nullity at any point.
	
	\vspace{.15cm}
	
	Let   $N : =  \pi ^{-1}(\bar N)$ be the lift of $\bar N$  to $\mathbb C^{n,1}$. Then, by Lemma \ref{lem:nullity43}, the nullity distribution $\mathcal N$ of $N$ coincides with the restriction to $N$ of the $2$-dimensional $\pi$-vertical distribution $\mathcal V$, i.e.,   $\mathcal V_q =\mathbb Cq$. 
	
	Since the signature of $N$ is the same as that of the ambient space, the normal space $\nu N$ is Riemannian. Moreover, by Corollary \ref{cor:NH-compare}, the normal holonomy group $\Phi (q)$ of $N$ at  $q$ acts as an irreducible Hermitian $s$-representation. 
	
	Let  $\zeta _q\in  \nu _q N$ be a small principal vector for the normal holonomy action and let 
	$(N)_{\zeta _q}$ be its associated  holonomy tube (possibly, by  making $N$ smaller around $q$).   Observe that $(N)_{\zeta _q} = (N)_{\zeta '}$, where $\zeta '$ is the normal parallel transport of $\zeta$ along any curve starting at $q$.

	Let $\eta _1 , \cdots , \eta _g$ be the curvature normals, with associated eigendistributions  $E_1, \cdots , E_g$, of the commuting family of shape operators of the isoparametric homogeneous submanifold 
	$\Phi(q)\cdot \zeta _q$ of $\nu _qN$ (see \cite {PT} or chapter 4 of \cite{BCO}). Moreover, the curvature normals are parallel in the normal connection of the orbit  $\Phi(q)\cdot \zeta _q$. We  regard such an orbit as a submanifold of the affine normal space $q + \nu _qN$. That is, we identify 
	$$\nu _qN \simeq q + \nu _qN \text{\ \ \ \ and \ \ \ } \Phi(q)\cdot \zeta _q \simeq q + \Phi(q)\cdot \zeta _q.$$
	
	Since $\Phi(q)$ acts irreducibly, $\Phi(q)\cdot \zeta _q$ is full in the normal space $\nu _qN$ and thus, the curvature normals span the normal space of  $\Phi(q)\cdot \zeta _q$ at any point of the orbit. Observe that the normal space to such an orbit coincides with the normal space of the holonomy tube 
	$(N)_{\zeta _q}$ (see 
 \cite[p.130]{BCO}). The integral manifold $S_i(x)$ of $E_i$ by $x\in \Phi(q)\cdot \zeta _ q$  is an extrinsic sphere, a so-called curvature sphere, of $\nu _qN$. One has that 
	\begin{equation}
	S_i(x)\subset	E_i(x)\oplus \mathbb R\,  \eta _i(x).
	\end{equation}
	 As in the case of Riemannian submanifolds of Euclidean space, every curvature normal of the holonomy orbit $\Phi(q)\cdot \zeta _q$ extends to a parallel normal field of $(N)_{\zeta _q}$ and its  associated autoparallel eigendistribution extends in a natural way to $(N)_{\zeta _q}$. We  denote such extensions by $\tilde \eta _i$ and $\tilde E_i$, $i=1, \cdots , g$. If  
	 $\mathrm{pr}: (N)_{\zeta _q} \to N$ denotes the projection, then the fibers, which are totally geodesic and invariant under  all shape operators of $(N)_{\zeta _q}$, are given by 
	 \begin{equation}\label{eq:fibers}
	 	\mathrm{pr}^{-1}(\{\mathrm{pr}(x)\}) = \mathrm{pr} (x) + \Phi(\mathrm{pr} (x))\cdot (x-\mathrm{pr}(x))\subset 
	 	\mathrm{pr} (x) + \nu _{\mathrm{pr} (x)} N. 
	 \end{equation}
	 Moreover, $\tilde \eta_i(x)$ is a curvature normal at $x$ of the orbit $\mathrm{pr}(x) + \Phi(\mathrm{pr}(x))\cdot (x-\mathrm{pr}(x))\subset \mathrm{pr}(x) + \nu _{pr(x)}N$. Furthermore,  $\tilde E_i(x)$ is the eigenspace associated with $\tilde \eta _i(x)$ (see \cite[Remark 7.3.1]{BCO}). Observe that $\Phi(\mathrm{pr} (x))\cdot (x-\mathrm{pr}(x))$ is identified with $\Phi(q)\cdot \zeta_q$ by means of the normal parallel transport in $N$ along any curve from $q$ to $\mathrm{pr}(x)$. 
	 
	 \begin{rem} \label{rem:tubeflat}
  The normal 
  space $\nu _x (N)_{\zeta _q}$  coincides with the normal space of $\Phi(\mathrm{pr}(x))\cdot (x-\mathrm{pr}(x))\subset \nu _{\mathrm{pr}(x)}N$. Hence,  $\tilde \eta _1(x), \cdots , \tilde \eta _g(x)$ span $\nu _x (N)_{\zeta _q}$ for all $x\in (N)_{\zeta _q}$. Then the principal holonomy tube $(N)_{\zeta _q}$ has a flat normal bundle (see \cite[Thm. 4.4.12]{BCO}). In particular, the normal field $\tilde{\zeta}$ of $(N)_{\zeta _q}$ defined by $\tilde{\zeta} (x) = \mathrm{pr}(x) -x$ is parallel and hence $N$ is a parallel focal manifold of the holonomy tube. Namely, 
	 \begin{equation}\label{eq:N=N}
	 	N =((N)_{\zeta _q})_{\tilde {\zeta}}
	 \end{equation}
	 
	\noindent  Note that 
	 $\langle\tilde\zeta , \tilde \eta _i\rangle =1$ for $i= 1, \cdots , g$. In particular, $\tilde \eta _i \neq 0$.
  \end{rem}
	
	Recall that the nullity distribution $\mathcal N$ of $N$ coincides with the vertical distribution $\mathcal V$. Since the perpendicular distribution to $\mathcal V$ in $N$ is Riemannian, then the commuting family of shape operators of $(N)_{\zeta _q}$  can be simultaneously diagonalized, with real eigenvalues functions. In fact, this follows  from the {\it tube formula} \cite [Lemma 3.4.7]{BCO}.  
	Namely, at any point $x\in (N)_{\zeta _q}$ there exist different curvature normals $0 = \tilde \eta _0(x), \tilde \eta _1 (x), \cdots , \tilde \eta _{d(x)} \in \nu _x(N)_{\zeta _q}$, $d(x) > g$, and orthogonal decomposition $T_x(N)_{\zeta _ q} = \tilde E_0(x) \oplus \cdots \oplus \tilde  E_{d(x)}(x)$ such that 
	$$\tilde A_{\psi \vert \tilde E_i(x)} = \langle \psi , \tilde\eta _i(x)\rangle Id_{\tilde E_i(x)}$$
	for all $\psi \in \nu _x(N)_{\zeta _q}$, where $\tilde A$ denotes the shape operator of $(N)_{\zeta _q}$. As for Euclidean submanifolds, in an open and dense subset $\Omega$ of $(N)_{\zeta _q}$, $d(x)$ is locally constant. Moreover, $\tilde  E_i$ is an integrable distribution and $ \tilde \eta _i $ is a smooth normal field. Since we are working locally, we may assume that $\Omega = (N)_{\zeta _q}$ and that $d= d(x)$ does not depend on $x$.  In our notation 
	$\tilde \eta _1 , \cdots , \tilde \eta _g$ are the above mentioned extensions of the curvature normals of the holonomy orbit, being  $\tilde E_1 , \cdots , \tilde  E_g$ their associated (autoparallel) eigendistributions. Namely, $\tilde E_1 , \cdots , \tilde  E_g$ are the  {\it vertical} eigendistributions of $\nu _x(N)_{\zeta _q}$, with respect to the projection 
	$\mathrm{pr} : \nu _x(N)_{\zeta _q} \to N$. 
 
 In general, the curvature normals $\tilde \eta _{g+1}, \cdots , \tilde \eta _{d}$ are not $\nabla ^\perp$-parallel and the eigendistributions 
	$\tilde E_{g+1} , \cdots , \tilde E_d$  are not autoparallel. One has that $T(N)_{\zeta _q}$ decompose orthogonally as 
	\begin{equation}
	T(N)_{\zeta _q} = \hat {\mathcal V} \oplus \mathcal H
	\end{equation}
where $\mathcal H = (\ker (\mathrm{d}\, \mathrm{pr}))^\perp$ is the  $\mathrm{pr}$-horizontal distribution of the holonomy tube $(N)_{\zeta _q}$ and $\hat {\mathcal V}= \tilde E_1 \oplus  \cdots  \oplus\tilde  E_g$ is the vertical distribution. 
\begin{rem}\label{rem:ufa}From the tube formula \cite [Lemma 3.4.7]{BCO} (see (\ref{eq:8393})) one obtains that
\begin{equation}\label{eq:nul50}
\tilde{\mathcal V} \subset \tilde E_0
\end{equation}
where $\tilde E_0$ is the nullity distribution of 
$(N)_{\zeta _q}$
 and $\tilde{\mathcal V}$ is the $\mathrm{pr}$-horizontal lift of the distribution $\mathcal V$ of $N$ (in particular, $\tilde A_{\tilde\zeta }(\tilde {\mathcal V})= 0$). Then 
 \begin{equation}\label{eq:8393} \mathcal V _{\mathrm{pr}(x)} =\mathrm{d}_x \mathrm{pr} (\tilde {\mathcal V}_x) = 
 (Id -\tilde A_{\tilde\zeta (x)})(\tilde {\mathcal V}_x) = 
 \tilde {\mathcal V}_x
\end{equation}
This implies that  the distribution $\tilde {\mathcal V}$ is constant, in the ambient   space $\mathbb C^{n,1}$, along  any fibre $S(x):=\mathrm{pr}^{-1}(\{\mathrm{pr}(x)\})$.
\end{rem}

 Let  $\psi \in \nu _{\tilde q}(N)_{\zeta _q}$ be generic in the sense that it is not perpendicular to some $\tilde{\eta} _i (\tilde q) -\tilde{\eta} _j (\tilde q)$, $i,j \in \{0, 1, \cdots ,d\}$, $i\neq j$. Then $\psi$ extends to a parallel normal field $\tilde{\psi}$ around $\tilde q$ that distinguishes all the eigenvalues functions $\lambda _i (\cdot): = \langle \, \cdot  \,  , \tilde \eta _i  \rangle$. However, we will be  interested in  some parallel normal fields $\tilde{\xi}$ that do not distinguish such eigenvalue functions. In particular,   in the case that 
 $\ker \tilde A_{\tilde \xi}$ is  bigger than $\tilde E_0$, around a generic point  where $\dim \ker \tilde A_{\tilde \xi}$ is constant and hence a distribution. 

Let now $\tilde{\xi}$ be a parallel normal field of 
$(N)_{\zeta _q}$. Since we work locally, we may assume that 
$\ker \tilde A_{\tilde \xi}$ has constant dimension and so, by Codazzi identity, $\ker \tilde A_{\tilde \xi}$ is an autoparallel distribution that is invariant, due to Ricci identity, by all the shape operators of $(N)_{\zeta _q}$. Since $\tilde{\mathcal V}\subset \ker \tilde A_{\tilde \xi}$, this distribution is pseudo-Riemannian and the orthogonally complementary distribution $\mathcal H^ {\tilde{\xi}}: =(\ker \tilde A_{\tilde \xi})^\perp$ is Riemannian. Let us consider the equivalence relation on $(N)_{\zeta _q}$ given by $x\underset{\tilde \xi}{\sim }y$ if there exists a $\mathcal H^ {\tilde{\xi}}$-horizontal curve that connects $x$ with $y$ (see \cite[p.224]{BCO}). About a generic point the  equivalence classes have all the same dimension.
By means of the horosphere  embedding $f$ (see Section \ref{sec:horosphere}) every  equivalence class $H^{\tilde \xi}(x):=[x]$ may be locally viewed as a (possible degenerate) holonomy tube around a focal Riemannian manifold. Thus,  we can apply the results of Section \ref{subs:2.3aaa}, after replacing $N$ by $M=f((N)_{\zeta _q})$. Observe that under these identifications $\ker \tilde A_{\tilde \xi} = \ker (Id -A_{\tilde \xi -\tilde v})$, where $A$ is the shape operator of $M$ and $\tilde v$ is the (umbilical) position vector field.   In particular, by Lemma \ref{lem:743}, the tangent space to any equivalence class $\mathcal T_x : =T_xH^{\tilde \xi}(x)$ is invariant under all   shape operators of $(N)_{\zeta _q}$. 

Before stating the next crucial result, we introduce some notation: let $\mathbf F$ be the foliation of $N$ given by the hypersurfaces obtained by the intersection of $N$ with  the family of pseudo-hyperbolic spaces
$H^{2n,1}_r$ (see (\ref{eq:pHip2})). The element of 
${\mathbf F}$ that contains $q\in N$ is denoted by $F(q)$. Let  $\tilde{\mathbf F}: =\mathrm{pr}^{-1}(\mathbf F)$ which is a  foliation   of $(N)_{\zeta_q}$ by hypersurfaces. The element of  $\tilde{\mathbf F}$ that contains $x$ is denoted by $\tilde F(x)$. Then:

\begin{Mlemma}\label{lem:main343} We are under the assumptions and notation of this section. Let $U$ be an open subset of  $(N)_{\zeta_q}$ such that for all $x\in U$  the equivalence classes $H^{\tilde \xi}(x)$ have the same dimension. 
	Then, for all $x\in U$, 
	\begin{enumerate}
		\item S(x):= $\mathrm{pr}^{-1}(\{\mathrm{pr}(x)\}) \subset H^{\tilde \xi}(x)$ (locally). 
		\item $H^{\tilde \xi}(x)$ is non-degenerate. 
		\item $H^{\tilde \xi}(x) = \tilde F(x)$ (locally). In particular, the foliation $\tilde{\mathrm {F}}$ does not depend on  $\tilde \xi$ (and its is called the {\it canonical foliation} of the holonomy tube).
	\end{enumerate}
\end{Mlemma}
\begin{proof}
	Part (1) follows with exactly the same arguments, relying on the Homogeneous Slice Theorem, as those used for Euclidean submanifolds in \cite[Section 7.3, p. 217]{BCO}.
	(see also \cite[p. 202]{CDO}.
	 In order to prove part (2), we will first prove that the induced metric on $H^{\tilde \xi}(x)$, if degenerate, is positive semi-definite with a one-dimensional degeneracy.  
	Let us consider the foliation $\mathbf F$ of $N$.  Note  that any leaf $F(p)$ of this foliation  is a pseudo-Riemannian hypersurface of $N$ with signature $1$.
	This foliation is perpendicular to  the position vector field $\vec v$. Observe that  $\vec v$ lies in the vertical distribution $\mathcal V$ and hence in the nullity distribution of $N$.  Let us consider the foliation $\tilde{\mathbf F} =\mathrm{pr}^{-1}(\mathbf F)$ by pseudo-Riemannian hypersurfaces of signature $1$ of 
	$(N)_{\zeta _q}$. 
	Let $\tilde v$ be the $\mathrm{pr}$-horizontal lift of $\vec v$. Then, by the tube formula \cite [Lemma 3.4.7]{BCO}, $\tilde v$ lies in the nullity distribution of $(N)_{\zeta_q}$ and hence in $\ker \tilde A_{\tilde \xi}$. Note that 
		$T_x\tilde F(x)= \tilde v_x^\perp$.
\noindent	Then \( H^{\tilde{\xi}}(x) \) lies in the leaf \( \tilde{F}(x) \) of \( \tilde{\mathbf{F}} \) through \( x \), which implies our assertion, since a linear subspace of a space of signature \( 1 \) is either Lorentzian or positive semi-definite.

	Let $\tilde \nu$ be the (autoparallel) distribution of $U$ which is perpendicular to the distribution $\mathcal T$ of tangent spaces of the equivalence classes $H^{\tilde \xi}(x)$  (see Section \ref{subs:2.3aaa}). 
	
	Assume that $H^{\tilde \xi}(x)$ is degenerate at $x$.   Then, 
	the intersection of $T_x H^{\tilde \xi}(x)\cap \tilde \nu _x$ is one-dimensional. Let $\psi \neq 0$  belong to this intersection. 
	Note that $\psi$ is an isotropic vector, i.e.  $\langle \psi , \psi\rangle =0$. From Lemma \ref{lem:743} the distribution $\mathcal T$ is invariant under all shape operators of $(N)_{\zeta_q}$ and in particular by  $\tilde A_{\tilde \zeta}$. Hence, $\tilde A_{\tilde \zeta _x\vert \mathcal T_x}= \hat A_{\tilde \zeta _x}$
	where $\hat A$ is the shape operator of $H^{\tilde \xi}(x)$ (see Remark \ref{rem:parallel-shape} for the definition of a parallel normal field to a degenerate submanifold, and its associated shape operator).  By the first part of this section, 
	$\tilde A_{\tilde \zeta}$ is diagonalizable, with real eigenvalues $\langle \tilde \zeta _x , (\tilde \eta _i)_x \rangle$, $i=0, \cdots , d$ (see the paragraph below Remark \ref{rem:tubeflat}). Then, 
	$\hat A_{\tilde \zeta _x}$ is diagonalizable with real eigenvalues. Since, from part (i), the distribution  tangent to the $\mathrm {pr}$-fibres, is 
	contained in $\mathcal T$, then the $1$-eigenspace of 
	$\hat A_{\tilde \zeta _x}$ coincides with the $1$-eigenspace $E_1^{\tilde \zeta}(x)$ of $\tilde A_{\tilde \zeta _x}$. By the last part of Remark \ref{rem:parallel-shape}, 
	$\hat A_{\zeta _x} \mathbb R\psi \subset \mathbb R\psi$ and hence $\psi$ is an eigenvector.  The only non-positive definite eigenspace of $\tilde A_{\tilde \zeta _x}$ is  $\ker \tilde A_{\tilde 
		\zeta _x}$ . Since $\psi$ is isotropic, we conclude that $\psi$ is a $0$-eigenvector, i.e. 
	$\hat A_{\tilde \zeta _x}\psi =0$. We regard now the isotropic vector $\psi$ as a vector perpendicular to $H^{\tilde \xi}$ at $x$,  and hence it extends to a parallel normal field $\tilde \psi$ of $H^{\tilde \xi}$ (see Lemma \ref{lem:744}). Let $v$ belong to  $$\tilde E_1^{\tilde \zeta}(x)= 
	\ker(Id- \tilde A_{\tilde \zeta _x} ) = 
	\ker(Id- \hat A_{\tilde \zeta _x})$$
	 and let $w=\hat A_{\tilde \psi _x}(v)$. Then, from Lemma \ref {lem:comm}, one has that 
		\begin{equation}\label{lap}\hat A_{\tilde \zeta _x}w = w + \lambda(v)\tilde \psi _x
		\end{equation} 
  for some scalar $\lambda (v)$ (we have used that the degeneracy of the metric of $H^{\tilde \xi}(x)$ has dimension $1$). Observe that the subspace $\mathcal T_x$ is invariant under the shape operator $\tilde A_{\tilde \zeta_x}$ of $(N)_{\zeta _q}$. Thus,   $\hat A_{\tilde \zeta_x} =(\tilde A_{\tilde \zeta_x})_{\vert \mathcal T}$ diagonalizes with real different eigenvalues
		 $\lambda _0 =0, \lambda _1 = 1 , \lambda _2,\cdots , \lambda _m$. Decompose $w= w_0+w_1+ \cdots + w_m$, where 
		$w_i$ is an eigenvector associated with  $\lambda _i$, $i=1, \cdots , m$. 
		
		 Then, by equation (\ref{lap}), since $\tilde \psi _x$ is a $0$-eigenvector, we conclude that $w=\hat A_{\tilde \psi _x}(v)$ is an $1$-eigenvector of  $\hat A_{\tilde \zeta_x}$. Then 
\begin{equation}
	\hat A_{\tilde \psi _x}\tilde E_1^{\tilde \zeta}(x)
	\subset \tilde E_1^{\tilde \zeta}(x)
\end{equation}
				and the same is true if one replaces $x$ for any arbitrary nearby $y\in 
				H^{\tilde \xi}(x)$. This implies that 
				$\tilde \psi$ is a parallel normal field of $S(x)= \mathrm{pr}^{-1}(\{\mathrm{pr}(x)\})$.
				
			Since $S(x)$ is a totally geodesic submanifold of $(N)_{\zeta_q}$ which is invariant under all shape operators, it is contained in the affine subspace

	 \begin{equation}  y +T_yS(x) \oplus  \nu _y (N)_{\zeta_q} \supset S(x)
	 	\end{equation}
	 for all $y\in S(x)$. The affine subspace 
	 $y +T_yS(y) \oplus  \nu _y (N)_{\zeta_q}$ does not depend on $y\in S(x)$ (observe that $S(x)=S(y)$). Observe that $\tilde \psi$ is both  perpendicular 
 and tangent to the (degenerate) equivalence class $H^{\tilde \xi}$. While the  latter condition implies that  it is perpendicular to $\nu(N)_{\zeta _q}$,  the first condition implies that it is perpendicular to the pr-fibers $S(x)$ (see part (i)).
	 Then $\tilde \psi$ is a constant field when restricted to  $S(x)$, since 
	 it is a parallel normal field which is perpendicular to an affine subspace that contains $S(x)$.
	 
	 Recall  that $\mathrm{pr}(y)  = y + \tilde \zeta (y)$ 
	 (see equality (\ref{eq:N=N})). Then 
	 \begin{equation}\label{eq:endpoint}
	 	\mathrm {d}_y(\mathrm {pr}) (\tilde \psi (y)) = (Id - \tilde{A}_{\tilde \zeta (y)})\tilde \psi (y)= \tilde \psi (y).
	 \end{equation}
	 	Since $\tilde \psi$ is constant along the fiber  $S(x)$ we obtain that the constant field $\tilde \psi _{\vert S(x)}$ projects down to the  vector $\tilde \psi (x)\in T_{\mathrm{pr}(x)}N$. Observe that  the union of the normal spaces  of $(N)_{\zeta _q}$ at different points of 
	 	$S(x)$ generates 
	 	$\nu _{\mathrm{pr}(x)}N$. Then, taking into account that 
	 	$\psi = \tilde \psi (y)$ belongs to the nullity $\tilde E_0(y)$ of $(N)_{\zeta _q}$ for any $y\in S(x)$, we obtain from the tube formula that $\tilde \psi (y) = 
	 	\psi$ belongs to the nullity subspace  
	 	$ \mathcal N_{\mathrm{pr}(x)}$ of $N$. Then the vector 
	 	$ \psi \in T_y(N)_{\zeta _q}$ is time-like. A contradiction since $ \psi$ is isotropic. This proves (2).
	 	
	 	(3) The inclusion $H^{\tilde \xi}(x) \subset  \tilde F(x)$ was proved inside the demonstration of part (2). The following arguments are similar to those  in  
	 	\cite[section 2]{CDO} (see also \cite[chap. 7]{BCO}. Let us consider the distribution $\tilde\nu$ perpendicular to the equivalence classes $H^{\tilde \xi}(y)$ and let $\Sigma (p)$ be the totally geodesic integral manifold of $\tilde\nu$ by $p$. Since the equivalence classes are non-degenerate by part (2), the
	 	same argument used in  the proof of  Proposition 2 (iv) of  \cite {CDO} shows that the equivalence classes are parallel manifolds of the ambient space. Namely, 
	 	$$ H^{\tilde \xi}(x) = (H^{\tilde \xi}(p))_{\mu _{(p,x)}}$$
	 	where $\mu _{(p,x)}$ is the  parallel normal field of $H^{\tilde \xi}(p)$ with $\mu _{(p,x)}(p) = x-p$ ($x\in \Sigma (p)$, near $p$). 
	 Let $\tilde E_1, \cdots , \tilde E _g$ be the autoparallel eigendistributions of $(N)_{\zeta _q}$, with associated parallel curvature normals $\tilde \eta _1, \cdots , \tilde \eta _g$, determined by the isoparametric full submanifolds $S(x)$
	 of the affine normal space $\mathrm {pr}(x) +
	 \nu _{\mathrm {pr}(x)}N$. By part (i), the restriction of $\tilde E_i$ to any 
	 $H^{\tilde \xi}(x)$ is tangent to this equivalence class. By the tube formula 
	 \begin{equation}\label{(iii)} \tilde \eta _i(x) = \frac {1}{1 - \langle (x-p),\tilde \eta _i (p)\rangle}\  
	 \tilde\eta_i (p), \ \ \ 
	 \end{equation}
 $x\in \Sigma (p)$  near  $p$, $i=1, \cdots , g$. Since $\tilde \eta _i$ has constant length,  we conclude that  $$\langle (x-p),\tilde \eta _i (p)\rangle = 0$$
or, more generally, 
\begin{equation}\label{2-2}\langle \mu _{(p,x)},\tilde \eta _{i \vert H^{\tilde \xi}(p) }\rangle = 0.
\end{equation}
	Since the curvature normals $\tilde \eta _1, \cdots , \tilde\eta _g$ associated to the $\mathrm {pr}$-fibres  generate, at any point,  the normal space of $S(p)$, regarded as a submanifold of $\nu _{\mathrm{pr}(p)}N$. Then, from (\ref{2-2}), $\mu _{(p,x)\vert S(p)}$ is a constant normal field along $S(p)$, regarded as a submanifold of the full ambient space $\mathbb C^{n,1}$. Then,  $x-p$ projects trivially to $\nu_p(N)_{\zeta _q}$, since it is spanned by $\tilde \eta _1(p), \cdots , \tilde \eta _g(p)$. Observe that $x-p$ is perpendicular to $H^{\tilde \xi}(p)$ at $p$, since it is the initial condition at $p$ of the normal field $\mu _{(p,x)}$. Since $x$ is arbitrary in $\Sigma (p)$, we obtain that 
	$\Sigma (p)$ is contained in the affine subspace 
	$p+ \tilde \nu _q = p + T_p\Sigma (p)$ and so it locally coincides with this subspace near $p$. This implies that $\tilde \alpha (\tilde \nu _p, \tilde \nu _p) = 0$, where $\tilde \alpha$ is the second fundamental form of $(N)_{\zeta _q}$. Since $\tilde \nu _p$ is invariant under all shapes operators of $(N)_{\zeta _q}$ at $p$, we obtain that $\tilde\nu _p$ is contained in the nullity of $\tilde \alpha$. Taking into account that     the parallel normal field  $\mu _{(p,x)}$ of $H^{\tilde \xi}(p)$ is constant along $S(p)$, one obtains  that $\Sigma (p)$ is a parallel affine subspace to $\Sigma (r)$ in the full ambient space, for all $r\in S(p)$ (locally). Then $\tilde\nu _p = \tilde \nu _r$ for all 
	 $r\in S(p)$, as linear subspaces. This implies, by the tube formula and the fact that the normal spaces of $\nu _rS(p)$, $r\in S(p)\subset \mathrm{pr}(p) + \nu _{\mathrm{pr}(p)}N $ span 
  	 $\nu _{\mathrm{pr}(p)}N$,  that $\tilde \nu _p$ belongs to the nullity of the second fundamental form $\alpha$ of $N$ at $\mathrm{pr}(p)$ (see \cite[chap.  7.3.2]{BCO}. Since the nullity of $\alpha$ is the distribution 
	 $y\to \mathbb Cy$, we obtain that $dim \,  \tilde \nu _p\leq 2$ and thus, the codimension of  $H^{\tilde \xi}(x)$ in $(N)_{\zeta _q}$ is at most $2$. Then the equivalence classes $H^{\tilde \xi}(x)$, since $H^{\tilde \xi}(x) \subset  \tilde F(x)$, locally coincide with $\tilde F(x)$ or have codimension $1$ in 
	 $\tilde F(x)$. In the first case we are done. In the second case, since $S(x)\subset H^{\tilde \xi}(x)$, the integrable foliation $\tilde {\mathcal T}$, given by the tangent spaces of the equivalence classes $H^{\tilde \xi}(x)$, projects down to an integrable distribution $\mathcal T := \mathrm{d\,}\mathrm{pr}(\tilde{\mathcal T})$ which (locally) coincides with the  distribution perpendicular to the vertical foliation $q \mapsto \mathbb C^*q$ of $N$. This contradicts Lemma \ref{lem:int3}. Thus, 
	 $H^{\tilde \xi}(x)$ coincides locally with 
	  $\tilde F(x)$.
	 	
	\end{proof}
	
 \begin{rem}\label{rem:971} We keep the notation and assumptions of this section. It was proved, inside the proof of Lemma \ref{lem:nullity43}, that 
 	$v\in T_zH^{2n,1}_r$ belongs to the nullity space of $N$ if and only if $v$ belongs to the nullity of the second fundamental form  of $F(z)= N \cap H^{2n,1}_r$ as a submanifold of 
 	$ H^{2n,1}_r$. Since the nullity of $N$ coincides with the distribution $y\mapsto \mathbb Cy$, we obtain that $\mathbb RJz$ coincides with the nullity of $F(z)$ as a submanifold of the umbilical submanifold  $H^{2n,1}_r\subset \mathbb C^{n,1}$, $r=\Vert z\Vert$. If  we regard $F(z)$ as a submanifold of $\mathbb C^{n,1}$, then 
 	we can decompose orthogonally the normal bundle 
 	into two parallel sub-bundles. Namely,  
 	$\nu F(z) = \nu_1 \oplus \nu _2$ where $\nu _1$ is one-dimensional, spanned by the position vector field, and $\nu _2$ is the normal bundle of $F(z)$ as a submanifold of $H^{2n,1}_r$. Then the distribution of $F(z)$ given by $x\to Jx$ is the common kernel of the family of shape operators 
 	$\{A_\mu : \nu \text{ is a section of } \nu _2\}$.  Observe that  the shape operator of the position vector field of $H^{2n,1}_r$ is minus the identity. 
 	
 Let $\tilde v$ be the  $\mathrm{pr}$-horizontal lift of $J\vec{v}$, where $\vec{v}$ is the position vector field of $N$. Observe that $\tilde v$ is time-like, and it is  tangent to any $\tilde F(x):= \mathrm {pr}^{-1}(F(\mathrm{pr}(x)))$. Moreover, by making use of the tube formula, we obtain that the one-dimensional distribution $\mathbb R\tilde v_{\vert \tilde F(x)}$
 is invariant under all the shape operators of $\tilde F(x)$. 
 Taking into account that the one-dimensional  bundle $\nu _1$ is time-like, we obtain  that the curvature normal $\tilde{\eta}_{g+1}$, associated with  the distribution $\mathbb R\tilde v_{\vert \tilde F(x)}$ 
 of $\tilde F(x)$, is 
 \begin{equation}\label{eq:cur1}
 	\tilde{\eta} _{g+1} = \frac{1}{r^2}\,{\vec{u}}_{\vert \tilde F(x)}
 \end{equation}
where ${\vec{u}}$ is the horizontal lift to $(N)_{\zeta_q}$ of the position (tangent) vector field $\vec{v}$ of $N$ and 
$r^2 = -\langle \vec{v}_{\mathrm{pr}(x)},  \vec{v}_{\mathrm{pr}(x)}\rangle$. Note, from the definition, that $F(z)= F(z')$, if $z'\in F(z)$ and 
$\langle z, z\rangle = \langle z', z'\rangle$. Moreover, $\tilde{\eta} _{g+1}$ is parallel in the normal connection of $\tilde F(x)$, as a Lorentzian submanifold of the ambient space $\mathbb C^{n,1}$. 

The labeling index  of $\tilde{\eta} _{g+1}$  is due to the fact that in our notation $\tilde {\eta} _1, \cdots , \tilde{\eta} _g$ are the parallel curvature normals associated to the vertical autoparallel distribution of $\tilde F(x)$
whose integral manifolds are isoparametric submanifolds of the ambient space  (see part (1) of Lemma  \ref{lem:main343}). The eigendistribution $\tilde E_{g+1}$, associated with $\tilde{\eta} _{g+1}$, could be bigger than $\mathbb R \tilde v$. In fact, it coincides with the restriction to $\tilde F(x)$ of the orthogonal complement of $\vec{u}$ in $\tilde E_0$, where 
$\tilde E_0$ is the nullity distribution of $(N)_{\zeta _q}$.
 \end{rem}
 
 \begin{lemma}\label{lem:coincides}
 	The local normal holonomy at $\mathrm{pr}(x)$ of $F(\mathrm{pr}(x))$, restricted to the orthogonal complement of the position vector $\vec v$,  coincides with the local normal holonomy group of $N$ at $\mathrm{pr}(x)$.
 	\end{lemma}
 \begin{proof} The proof is the same as that of Lemma 7.3.5 (i).
 \end{proof}
	\section{The geometry of the equivalence classes }\label{sec:isop1}
	
	To ensure clarity, we begin by summarizing the main results of the previous section, explaining them in some detail. Recall that we work locally, and our results---though not always stated explicitly---hold in a neighborhood of a generic point  of the submanifolds involved.

 If $\tilde \xi$  is a parallel normal field of $(N)_{\zeta _q}$, then 
	$H^{\tilde \xi}(x)\subset (N)_{\zeta _q}$ is the equivalence class of $x$, where $x\sim y$ if there exists a curve perpendicular to $\ker \tilde A_{\tilde \xi}$ connecting $x$ with $y$.  Then 
	$H^{\tilde \xi}(x)$ is a hypersurface of 
	$(N)_{\zeta _q}$ that  (locally) coincides with 
	$\tilde F(x) = \mathrm{pr}^{-1}(F(\mathrm{pr}(x)))$, where 
	$F(\mathrm{pr}(x))= N\cap H_r^{2n,1}$ and $r^2= -\langle \mathrm{pr}(x), \mathrm{pr}(x)\rangle$ (see Lemma \ref{lem:main343}). One has that 
	$\tilde F(x)$ is invariant under all shape operators of $(N)_{\zeta _q}$ (see Lemma \ref{lem:743} and the paragraphs previous to Lemma \ref{lem:main343}). Moreover, the normal bundle of $\tilde F(x)\subset \mathbb C^{n,1}$ splits as the orthogonal sum of the following parallel and flat  subbundles
\begin{equation}\label{eq:810}
	\nu \tilde F(x) = \tilde \nu _1\oplus \tilde \nu _2 \, ,
\end{equation}
	where $\tilde \nu _1 = \mathbb R {\vec{u}}_{\vert \tilde F(x)}$, $\tilde \nu _2 = (\nu (N)_{\zeta _q})_{\vert \tilde F(x)}$ and $\vec{u}$ is the $\mathrm{pr}$-horizontal lift of the position vector field $\vec{v}$ of $N$. One has that both $\tilde F(x)$ and its normal bundle are Lorentzian.
	 Moreover, the commuting symmetric family  of  shape operators $\{\tilde A_\mu\}$ of $\tilde F(x)$
	diagonalize  simultaneously  with real eigenvalues.
	In fact, the eigendistribution $\tilde E_{g+1}$ associated to the  parallel section $\tilde \eta _{g+1}$ is non-zero and contains the timelike vector $\tilde v$ (see last part of  Remark \ref{rem:971}). Since 
	$\tilde F(x)$ is Lorentzian, $\tilde E_0$ is non-degenerate and its orthogonal complement is a Riemannian distribution. This implies our assertion.
	
	Let us finally recall that $\tilde F(x)$ contains the fibre $\mathrm{pr}^{-1}(\{\mathrm{pr}(x)\})$ (see Lemma \ref{lem:main343}).

\begin{rem}\label{rem:sinNul} The  nullity of $\tilde F(x)$ is trivial, as a submanifold of $\mathbb C^{n,1}$.  In fact, let $\tilde \zeta$ be the 
	parallel normal  vector field of $(N)_{\zeta _q}$ such that $\mathrm {pr}(y) = y +  \tilde \zeta (y)$. Then, by the tube formula, see (\ref{eq:endpoint}), 
	$\mathrm {d\, pr}(\vec u_y)= (Id - \tilde A_{\zeta})\vec u_y = \vec u _y = \vec v _{\mathrm{pr}(y)}$, as vectors of the ambient space $\mathbb C^{n,1}$. Taking into account that the position (normal) vector field of $F(\mathrm{pr}(y))$ is umbilical, we obtain 
	$$-Id = A_{\vec v (\mathrm{pr}(y))}= 
	\tilde A_{\vec u _y}\big((Id- 
	\tilde A_{\tilde \zeta _y})_{\vert \mathcal H}\big)^{-1},$$
	where $\mathcal H$ is the $\mathrm {pr}$-horizontal distribution of $\tilde F(x)$. This shows that 
	$\tilde A_{\vec u _y\vert \mathcal H}$ has no kernel. Since the $\mathrm{pr}$-fibres are irreducible isoparametric submanifolds, the family of shape operators 
	$\tilde A_\psi$, $\psi \in (\tilde \nu _2)_y$, restricted to the $\mathrm{pr}$-vertical distribution $\mathcal H^\perp$, have no common kernel. The previous observations imply our assertion.
\end{rem}	

\vspace{.2cm}

Let, keeping  the notation of Section \ref{sec:NT1}, and Remark \ref{rem:sinNul},  
$\tilde \eta_1, \cdots , \tilde \eta _d$  ($d\geq g+1$)   be the curvature normals of $\tilde F(x)$ with associated eigendistribution $\tilde E_1, \cdots , \tilde E_d$, which are integrable due to the Codazzi identity  (perhaps in a neighbourhood of a point close to $x$). Recall that $\tilde \eta_1, \cdots  \tilde \eta _{g}$ are parallel, and $\eta _{g+1}$ is also parallel due to Remark \ref{rem:971}.   Moreover, all eigendistributions are Riemannian  with the exception of  $\tilde E_{g+1}$ which is Lorentzian. 

Asumme that  $\tilde \eta _i$ is a parallel curvature normal, then any integral manifold $S(y)$ of $\tilde E_{i}$ is totally geodesic in $\tilde F(x)$. Moreover, $S_i(y)$ is an umbilical submanifold of the ambient space $\mathbb C^{n,1}$ which is contained  in the affine space 
$$ 	y +  \tilde E_{i}(y)+ \mathbb R \tilde \eta _i (y).$$

({\it a}) If $\tilde \eta _i (y)$ is spacelike, and so $i\neq g+1$, then  $S_i(y)$ is an open subset of the sphere  of the Euclidean space 
$y +  \tilde E_{i}(y)+ \mathbb R \,\tilde \eta _i (y)$ c
with center $c$ and radius $\rho$ given by 
\begin{equation}\label{eq:839} c= y + \frac {1}{\langle \tilde \eta _i(y), \tilde \eta _i(y)\rangle}\, \tilde \eta _i\, ,
\ \ \ \ \ \ \
\rho = \frac {1}{\sqrt{\langle \tilde \eta _i(y), \tilde \eta _i(y)\rangle}}.
\end{equation}
In this case $S_i(y)$ is called a {\it curvature sphere}

({\it b}) If $\tilde \eta _i$ is lightlike, and so  $i > g+1$, then $S_i(y)$ is a horosphere of an appropriate real hyperbolic space. In fact, there always exist a timelike $z\in \nu_y\tilde F(y)$ such that 
$\langle \tilde \eta _i , z\rangle =1$. Let 
$-r^2 = \langle z , z\rangle$ and let 
$$H_r^k = \{w\in y + \tilde E_i(y) \oplus \mathbb R \, \tilde \eta _i(y) \oplus \mathbb R \, z:  
\langle w-(y+z), w-(y+z)\rangle = -r^2\}^o$$
where $k=\dim \tilde E_i(y) + 1$ and $(\ )^o$ denotes the connected component by $y$. Then $S_i(x)$ is an open subset of the horosphere defined by 
$$\bigg(y + \tilde E_i(y) \oplus \mathbb R \, \tilde \eta _i(y)\bigg)\cap  H_r^k\, .$$

({\it c}) If $\eta _i$ is timelike, $i\neq g+1$, then $S_i(x)$ is an open subset of the hyperbolic space  of $L_i(x)$ defined by 
	$$H_r^{k_i} = \{X\in x +  E_i(x) \oplus \mathbb R \,  \eta _i(x):  
	\langle X-c, X-c)\rangle = -r^2 \}^o$$, where
	 	$-r^2= \langle x-c,x-c\rangle $, and 
 	 $c$ has the same expression as in (a).
     \vspace{.15cm}

({\it d}) If $i=g+1$, there are two cases: 

\vspace{.15cm}

\noindent {\it $\bullet$ } $\dim \tilde E_{g+1}= 1$. Then, by Remark \ref{rem:971}, $\tilde E_{g+1}= \mathbb R \tilde v$. In this case $S_i(y)$ is an open subset ot the (compact)  anti-circle 
of the negative definite affine plane  
$y + \tilde E_{g+1}(y) \oplus \mathbb R \tilde \eta _{g+1}(y)$ of center 
$$ c = y + \frac {1}{\langle \tilde \eta _{g+1}(y), \tilde \eta _{g+1}(y)\rangle}\, \tilde \eta _{g+1}$$ 
and given by the equation 
$$\langle w-c, w-c\rangle = 
\frac {1}{{\langle \tilde \eta _{g+1}(y), \tilde \eta _{g+1}(y)\rangle}}.$$

\noindent {\bf $\bullet$} $\dim \tilde E_{g+1}> 1$. One gets the same formulas as in the previous case.  But, instead of an anti-circle one obtains a pseudo-hyperbolic space.

\vspace{.15cm}

\subsection{The  parallel focal set} 
\

We keep the notation and assumptions of this section.  
In order to simplify the exposition we introduce the following notation:
$$\tilde F := \tilde F(x), \ F:= F(\mathrm{pr}(x)).$$

We next discuss some standard facts, or definitions, that are well-known in a Euclidean ambient space and extend to our setting with straightforward modifications.

The {\it affine normal space}  of $\tilde F$ at $y$ is the affine subspace 
$$y + \nu _y \tilde F\subset \mathbb C^{n,1}\simeq \mathbb R^{2n,2}.$$

The {\it affine focal hyperplane} $\tilde \Sigma _j(y)\subset y + \nu _y \tilde F$  associated to 
$\tilde \eta _j(y)$ is 
$$\tilde \Sigma _j (y): = y + \tilde H_j(y), $$
where $\tilde H_j(y)$ is the linear hyperplane of $\nu _y \tilde F$ 
defined by the equation 
$\langle \tilde\eta _j(y), \cdot \,  \rangle = 1$ ($j= 1, \cdots , d$). The {\it focal set} at $y$  is defined by 
$$\cup _{j=1}^d \, \tilde \Sigma _j(y)$$
and the {\it parallel focal set} at $y$ is defined by 
$$\cup _{i\in I}\,  \tilde \Sigma _i (y)$$
where 
\begin{equation}\label{eq:I12}
I=\{i: \tilde \eta _i \text{ is parallel, } 1\leq i\leq d\}.
\end{equation}

Let $\xi$ be a parallel normal field of $\tilde F$ such that the parallel manifold $\tilde F_\xi$ is not singular, i.e., 
$I-\tilde A_\xi$ is never singular (perhaps after shrinking \( \xi \) and \( F \)). Equivalently, 
$\langle \xi_y, \tilde \eta _j(y)\rangle \neq 1$, for all $y\in \tilde F$, $j=1, \cdots , d$. Let $f$ be  the parallel map $y \overset{f}{\to} y + \xi (y)$. Then affine normal spaces  $y + \nu _y\tilde F$ and 
$f(y) + \nu _{f(y)}\tilde F_\xi$ do coincide.
Moreover, the focal set of $\tilde F$ at $y$ coincides with the focal set of 
$\tilde F_\xi$ at $f(y)$. In fact, this is a consequence of the tube formula that relates the shape operators of parallel manifolds (see  Lemma 3.4.7 and Proposition 4.4.11 of \cite{BCO}). One has that $f$ maps the  eigendistribution $\tilde E_j$ of $\tilde F$ into an eigendistribution $\tilde E_j^\xi := f_*(\tilde E_j)$
of $\tilde F_\xi$. Moreover, form the tube formula, the curvature normal $\tilde \eta _j^\xi$ associated to 
$\tilde E_j^\xi$  is given by 
$$\tilde \eta _j^\xi(f(p))=\frac {1}{1-\langle \tilde \eta _j(y), \xi(y)\rangle}
\tilde \eta _j(y).$$
Observe that $\tilde E_j$ is Riemannian if and only if $\tilde E_j^\xi$ is Riemannian. Moreover, 
$\tilde \eta _j$ is parallel if and only if $\tilde \eta _j^\xi$ is parallel. Then the parallel focal set of $\tilde F$ at $y$ coincides with that of $\tilde F _\xi$ at $f(y)$. Observe that $f$ maps a curvature sphere $S_i(y)$ into a curvature sphere $S_i^\xi(f(y))$; see (\ref{eq:839}).

\subsection{The Coxeter group}\

We keep the notation and assumptions of this section. 

    We keep the notation and assumptions of Sections 3 and 4. In particular, we assume that the index of relative nullity of $\bar N$ is zero. This implies that the normal holonomy of $N$ acts irreducibly. Let us further assume that the normal holonomy is not transitive on the unit sphere of the normal space. This implies, in particular, that the dimension of the normal space of $\tilde F$ is at least $3$.

The next main tools are inspired by Terng's construction of the Coxeter group of an isoparametric submanifold \cite[Section 6.3]{PT} (see also \cite[Section 4.2]{BCO}). We may assume, since we work locally, that $\tilde F$ is simply connected and so  $\nu \tilde F$ is globally flat. Let $y,y'\in \tilde F$ and let $\tau _{y,y'}: \nu _y\tilde F 
\to \nu _{y'}\tilde F$ be the parallel transport with respect to  the normal connection. Let $\tilde \tau _{y,y'}:
y + \nu _y\tilde F \to y' + \nu _{y'}\tilde F$ 
 be the so-called {\it affine parallel  transport}. Namely, $\tilde\tau _{y,y'}(y)=y'$ and $\mathrm{d}_y\tilde \tau _{y,y'}= 
 \tau _{y,y'}$. The affine parallel transport  maps  parallel focal hyperplanes into parallel focal hyperplanes. That is, for any $i\in I=\{i: \tilde \eta _i \text{ is parallel, } 1\leq i\leq d\},$
 $$\tilde \tau _{y,y'}\big(\tilde \Sigma _i (y)\big) = \tilde \Sigma _i (y')$$
 and hence the affine parallel transport maps parallel focal sets into parallel focal sets: 
 \begin{equation}\label{eq:A78}
 	\tilde \tau _{y,y'}\big(\cup _{i\in I} \tilde \Sigma _i (y)\big) = \cup _{i\in I} \tilde \Sigma _i(y')
 	\end{equation}

 Let 
 \begin{align}\label{eq:I014}
 	I_0:&=\{i\in I: \text{ the integral manifolds of } \tilde E_i \text{ are curvature spheres}\} \\ 
 	\nonumber 
 	&= \{i\in I:  \tilde E_i \text{ is Riemannian }
 	\text{ and } \tilde \eta _i \text{ is spacelike}\}
  	\end{align}
 
 		 Equivalently,  $\tilde E_i$ is Riemannian, and  $\tilde \eta _i$ is parallel and spacelike. Observe that $\{1, \cdots , g\}\subset I_0$.  Then
 \begin{equation}\label{eq:B78}\tilde \tau _{y,y'}\big(\tilde \Sigma _{i_0} (y)\big) = \tilde \Sigma _{i_0} (y'), \text{ for any  } i_0\in I_0.
 	\end{equation}
 Let  $i_0\in I_0$ and assume that the curvature spheres 
 $S_{i_0}(y)$, $y\in \tilde F$, are complete. Let, for $y\in \tilde F$, $\tilde y\in S_{i_0}(y)$ be the antipodal point of $y$. Then $\xi _{i_0}$, defined by $\xi _{i_0}(y)=\tilde y -y$ is a parallel normal field and $\tilde F_{\xi _{i_0}}= \tilde F$. In fact,  $\xi _{i_0}= \frac{2}{\langle \tilde \eta _{i_0},\tilde \eta _{i_0}\, \rangle}\tilde  \eta _{i_0}$.  
 
   The affine parallel transport $\tilde \tau _{y,\tilde y}$ may be achieved by parallel transporting along a curve in $S_{i_0}$ from $y$ to $\tilde y$. It turns out that this parallel transport coincides with the reflection $R_{i_0}$ in the focal affine hyperplane 
   $\tilde \Sigma _{i_0}(y)$ (see \cite[Section 4.2.2]{BCO}). The affine normal spaces  
   of $\tilde F=\tilde F_{\xi _{i_0}}$ at $y$ and $\tilde y$ coincide. Moreover, 
   \begin{equation}\label{eq:C78}\cup _{j\in J} \, \tilde \Sigma _j(y) = 
   \cup _{j\in J} \, \tilde \Sigma _j(\tilde y),
   \end{equation}
   where $J$ is any  of the following sets: $\{1, \cdots , d\}$, $I$, $I_0$. In fact, for $J=\{1, \cdots , d\}$ the equality (\ref{eq:C78}) is true since $\tilde F$ is a parallel manifold to itself. For $J=I, \, I_0$ it is a consequence of equalities 
   (\ref{eq:A78}) and (\ref{eq:B78}). 
   
   By the previous discussions we obtain that 
   \begin{equation}\label{eq:D78}
   	R_{i_0}\big(\cup _{j\in J} \, \tilde \Sigma _j(y)\big)= \cup _{j\in J} \, \tilde \Sigma _j(y)
   	\end{equation}
   
   In particular, 
   \begin{equation}\label{eq:E78}
   	R_{i_0}\big(\cup _{i\in I_0} \, \tilde \Sigma _i(y)\big)= \cup _{i\in I_0} \, \tilde \Sigma _i(y)
   \end{equation}

If the curvature sphere $S_{i_0}(y)$ is not complete, we can use an argument, used by Terng  in the proof of Theorem 3.4 in \cite{T}, in order to extend locally $\tilde F$ and  so that the curvature sphere $S_{i_0}(y)$ is complete. We now sketch this argument.   Let us consider the parallel focal manifold $F_\xi$, where $\xi = 
\frac{1}{\langle \tilde \eta _{i_0}, \tilde \eta _{i_0}\rangle} \tilde \eta _{i_0}$. Perhaps, by  passing before to a nearby generic parallel manifold to $\tilde F$ so that $\langle \tilde\eta_{i_0}, \tilde\eta_{j}\rangle =1 $ if and only if $j= i_0$ ($j= 1\cdots d$). Then consider the normal (parallel) subbundle $B$ over 
$\tilde F_\xi$ given by $B_{z} =\tilde E_{i_0}(y)\oplus \mathbb R\tilde \eta _{i_0}(y)$, where $\pi (y):= y + \xi (y) =z$. This subspace does not depend on $y$ such $\pi (y)=z$. In contrast to the framework in Terng's proof, not all curvature normals are necessarily parallel. Consequently, in our context, we have to consider  the complete sphere bundle $SB$ of $B$ of radius $\beta \langle \xi ,
\xi \rangle^{1/2}$, $\beta >0$ small. The image of $B$ under the normal exponential map of $\tilde F_\xi$ is the desired extension of $\tilde B$.  

By equation (\ref{eq:fibers}), $S(y)= \mathrm{pr}^{-1}(\{\mathrm{pr}(y)\})\subset \mathrm{pr} (y) + \nu _{\mathrm{pr} (y)} N$ is an irreducible isoparametric submanifold. Then 
$$\{R_{i\vert y + \nu_y(N)_{\zeta _q}}: 1\leq i \leq g\}$$ 
generates a (finite) Coxeter group $W$ which acts irreducibly on the affine normal space  
	$y + \nu_y(N)_{\zeta _q}$. Moreover, $\mathrm{pr}(y)$ is the only fixed point in such a space. Taking into account that 
	 $y + \nu_y(N)_{\zeta _q}$ is an affine hyperplane of   $y + \nu_y\tilde F$, 
 we obtain that 
 $$\cap _{i\in I_0}\tilde\Sigma (y) = \mathrm{pr}(y) + \mathbb R \vec{u}_y$$
 and hence the line 
 \begin{equation}\label{eq:F78} 
 	\mathrm{pr}(y) + \mathbb R \vec{u}_y
 	\end{equation}
 	  is the fixed set of $W$ acting on $y + \nu _y\tilde F(y)$. 
 
 Let $\tilde W$ be the group of affine transformations of $y + \nu_y\tilde F$ generated by the set 
 $\{R_i:i\in I_0\}$.   As for the case of Euclidean reflections, $\tilde W$ is a finite group. In fact, observe first that $R_i= R_i^{-1}$. If we now write $g\in \tilde R$ as a word of minimal length $g=R_{i_1} \cdots R_{i_m}$ then all factors are different due to the commutation law $R_iR_jR_i = R_k$ for some $k$.  Since $\tilde W$ is finite, it has a fixed point $p$. Such a point must belong  to the line $\mathrm{pr}(y) + \mathbb R \vec{u}_y$, since $W\subset \tilde W$ (see (\ref{eq:F78})). We will identify $\tilde W$ with a linear group with center $p$.
 Assume that   $\tilde W$ acts irreducibly on 
 $y + \nu _y\tilde F$. Then, since $\tilde W$ is finite there exits a positive definite scalar product $(\ , \ )$ that is invariant by $\tilde W$.  By the Schur lemma, the Lorentz inner product of $y + \nu _y\tilde F$ is a multiple of $(\ , \ )$. A contradiction since $\dim \nu _y\tilde F\geq 2$.
 Thus, the action of $\tilde W$ is reducible. From the fact that  $W\subset \tilde W$ acts irreducibly on the hyperplane 
 $y + \nu_y(N)_{\zeta _q}$ of  $y + \nu _y\tilde F$ 
 it follows that the unique non-trivial irreducible subspaces of $\tilde W$ are $y + \nu_y(N)_{\zeta _q}$ and $\mathrm{pr}(y) + \mathbb R \vec{u}_y$. Moreover, since the finite group $\tilde W_{\vert y + \nu_y(N)_{\zeta _q}}$ must have a fixed point, and the only fixed point of $W$ is such a space is 
 $\mathrm{pr}(y)$, we conclude that $\mathrm{pr}(y)$ is a fixed point of $\tilde W$. Then the parallel normal vector field $y\overset {\tilde \zeta} {\mapsto} y -\mathrm{pr}(y)$ satisfies that 
 $$\ker(Id - \tilde A_{\tilde \zeta (y)}) \supset 
 \oplus _{i\in I_0}\tilde E_i(y).$$
On the other hand, since $\tilde F_{\tilde \zeta}  =F $,  it follows that 
 $$\ker(Id - \tilde A_{\tilde \zeta (y)}) = 
 \oplus _{i=1}^ g\tilde E_i(y)$$
 (see (\ref{eq:N=N}). Then $\tilde W =W$ and 
 \begin{equation} \label{eq:I_0=}
 	I_0= \{1, \cdots , g\}.
 	\end{equation}
 
 \begin{cor}\label{cor:vones} The only curvature spheres of $\tilde F$ are the vertical ones associated to the isoparametric fibers of $\tilde F \overset{\mathrm{pr}}{\to} F$. \qed
 \end{cor}
	
 The above corollary together with the following  lemma will be crucial for our purposes.
 
 \begin{lemma}\label{lem:novones} Not all curvature normals of $\tilde F$ are parallel.
 \end{lemma}
\begin{proof} Assume that all curvature normals are parallel. Observe, keeping the notation of previous  sections, that there should exist  $i>g+1$ such that $\tilde \eta _i$ is not a scalar multiple of $\tilde \eta _{g+1}$. If not, $\ker \tilde A_{\tilde \zeta}$ and $\ker (Id - \tilde A_{\tilde \zeta})$ would be two (orthogonally) complementary non-degenerate totally geodesic distributions which are   invariant by  all shape operators. The affine subspace generated by any  integral manifold of $\ker (Id - \tilde A_{\tilde \zeta} )$ is Euclidean and hence non degenerate.  Then, by Moore's lemma (see e.g.  \cite[Lemma 2]{Wi} and \cite[Corollary 1.7.4]{BCO}),  $\tilde F$ locally splits and hence $F$ locally splits. The flat part  $\nu _0F$ of  normal bundle $\nu F$  has dimension $1$, and locally $N = \cup _{\xi}F_\xi$, where $\xi$ is a small parallel section of  $\nu _0F$. Then $N$ locally splits which is a  contradiction, since the normal holonomy of $N$ acts irreducibly (see the begining of this section).
	
 Let  $\tilde \eta _i$, $i>g$ be
  such that it is not a scalar multiple of $\vec u$. Then $i \neq g+1$
	 and thus, $\tilde E_i$ is Riemannian (see the paragraph below Remark \ref{rem:sinNul}). If $\tilde \eta _i$ is spacelike then the integral manifolds of $\tilde E_i$ are curvature spheres (see (\ref{eq:I014})). This contradicts Corollary \ref{cor:vones}. Then $\tilde \eta _i = \lambda \vec u + \tilde \mu $, where $\lambda \neq 0$ and $\tilde \mu \neq 0$ is the projection of $\tilde \eta _i$ to the parallel normal subbundle $\tilde \nu _2 = \vec u ^\perp = \nu (N)_{\zeta _q\, \vert \tilde F}$ (see (\ref{eq:810})). Let $j\in \{1, \cdots , g\}$ and let $\tilde \psi$ be a parallel section of $\tilde \nu _2$ such that $\langle \tilde \psi , \tilde \eta _j\rangle = 1 =\langle \tilde \psi , \tilde \mu \rangle$. In fact define $\tilde \psi$ as an appropriate scalar multiple of a non-zero parallel section of the bundle 
	 $\mathbb R \tilde \eta _j \oplus \mathbb R\tilde \mu$ which is perpendicular to 
	 $\tilde \eta _j - \tilde \mu$. Then 
	 $\tilde E := \ker (Id-\tilde A_{\tilde \psi})$, which is a direct sum of eigendistributions, contains 
	 $\tilde E_i\oplus \tilde E_j$ and does not contain $\tilde E_{g+1}$. Then any integral manifold $\tilde S(y)$ of the autoparallel distribution $\tilde E$ 
        is a Riemannian isoparametric submanifold of the affine Lorentzian subspace 
       $y +  \tilde E(y) \oplus \nu _y\tilde F$.
       Since $\tilde \eta _j$ is  spacelike and $\eta _i$ is not so, then  $\langle \tilde \eta _i ,\tilde \eta _j\rangle =0$  by 
       Proposition \ref{prop:sl-nsl}. Since $\{\eta _j: 1\leq j \leq g\}$  generates $\tilde \nu _2$, then $\eta _i$ lies in $\tilde \nu _1 = 
       \mathbb R\, \vec u$. A contradiction that implies the lemma. 

	\end{proof} 

 \vspace{.15cm}

 \begin{proof}[Proof of Theorem \ref{thm:MainTheorem1}] 
 
 By Corollary \ref{cor:NH-compare}, the normal holonomy group of $N$ acts irreducibly on the normal space. Let $N$ be the lift of $\bar N$ to $\mathbb C^{n,1}$, and let $\bar\xi$ be a parallel normal field to $N$ such that $\ker A_{\tilde \xi}$ has constant dimension. Then, according to Lemma \ref{lem:main343} (3), 
$H^{\tilde \xi}(x) = \tilde F(x)$, and so it does not depend on $\tilde \xi$. By Lemma \ref{lem:743}, $T_y\tilde F (x)$ is invariant under all shape operators of $(N)_{\zeta _q}$. Moreover, $\ker A_{\tilde\xi\vert TF(x)}$ is invariant under all the shape operators of $(N)_{\zeta _q}$. Let $S^{\tilde \xi}(y)$ be the  total geodesic and non-degenerate integral manifold of $\ker A_{\tilde\xi\vert TF(x)}$ by $y\in \tilde F (x)$. Then $S^{\tilde \xi}(y)$ is (locally) the orbit of a weakly polar action. Namely, via the horosphere embedding, it coincides with the normal holonomy orbit of an appropriate focal manifold. Since  $\tilde F(x)$ has flat normal bundle, and its family of shape operators are simultaneously diagonalizable, with real eigenvalues, the same is true for $S^{\tilde \xi}(y)$. Moreover, since $S^{\tilde \xi}(y)$ is the orbit of a weakly polar action, it follows that the curvature normals of $S^{\tilde \xi}(y)$ are parallel in the normal connection. As for the curvature normals associated to the fibers of $(N_{\tilde \zeta _q})$, any curvature normal of $S^{\tilde \xi}(y)$ extends to a parallel curvature normal of $\tilde F(x)$ (see \cite[sec.7.1]{BCO}). If the normal holonomy group of $N$ is not transitive we can find, as in \cite[sect. 7.4]{BCO}, parallel normal fields $\tilde \xi$, $\tilde \xi '$ of $\tilde F (x)$ such that $\ker A _{\xi} + \ker A _{\xi '}$ contains the horizontal distribution of $\tilde F (x)$. Then, Proposition 7.36 of \cite{BCO} applies with the same proof to show that any curvature normal of $\tilde F(x)$ is parallel in the normal connection. This contradicts Lemma \ref{lem:novones}, proving  that the normal holonomy must be transitive.
 \end{proof}

\end{document}